\documentclass[11pt]{amsart}

\usepackage{amssymb}
\usepackage{latexsym}
\usepackage{amsmath}
\usepackage{amsthm}
\usepackage{color}
\usepackage{hyperref}
\usepackage{ulem}

\newcommand{\R}{\mathbb{R}}
\newcommand{\RR}{\mathcal{R}}
\newcommand{\N}{\mathbb{N}}

\newcommand{\e}{ \epsilon  }
\newcommand{\la}{ \lambda  }
\newcommand{\ds}{ \displaystyle  }

\newcommand{\Om}{\Omega}
\newcommand{\lf}{\left}
\newcommand{\rg}{\right}
\newcommand{\ml}{\mathcal}
\newcommand{\fr}{\partial}
\newcommand{\lap}{\Delta}
\newcommand{\de}{\delta}
\newcommand{\grad}{\nabla}
\newcommand{\ti}{\tilde}
\newcommand{\sm}{\setminus}

\newcommand{\pd}[2]{\frac{\partial #1}{\partial #2}}

\theoremstyle{plain}
\newtheorem{defi}{Definition}[section]
\newtheorem{prop}[defi]{Proposition}
\newtheorem{teo}[defi]{Theorem}

\newtheorem{lema}[defi]{Lemma}
\newtheorem{claim}{Claim}

\theoremstyle{definition}

\theoremstyle{remark}

\numberwithin{equation}{section}
\begin{document}
	
	\title[Sign-changing solutions]{Sign-changing solutions for the sinh--Poisson equation with Robin Boundary condition}
	
	\author[P. Figueroa]{Pablo Figueroa}

	\author[L. Iturriaga]{Leonelo Iturriaga}
	
	\author[E. Topp]{Erwin Topp}
	
	\address{P. Figueroa\hfill\break\indent
		Instituto de Ciencias F\'{\i}sicas y Matem\'aticas, Facultad de Ciencias, Universidad Austral de Chile, Campus Isla Teja,
		Valdivia, Chile.} \email{{\tt pablo.figueroa@uach.cl}}
	
	\address{L. Iturriaga\hfill\break\indent
		Departamento de Matem\'{a}tica, Universidad T\'ecnica Federico Santa Mar\'{\i}a
		\hfill\break\indent  Casilla V-110, Avda. Espa\~na, 1680 --
		Valpara\'{\i}so, Chile.}
	\email{{\tt leonelo.iturriaga@usm.cl}}
	
	\address{E. Topp\hfill\break\indent
		Departamento de Matem\'{a}tica y C.C., Universidad de Santiago de Chile, Casilla 307, Santiago, Chile.}
	\email{{\tt erwin.topp@usach.cl}}
	\date{\today}
	
	\begin{abstract}
Given $\epsilon \in (0,1)$ and $\la > 1$, we address the existence of solutions for the Sinh-Poisson equation with Robin boundary value condition 
			$$
			\begin{cases}
				\Delta u+\epsilon^2 (e^{u} - e^{-u})=0 &\mbox{in }\Omega\\ \frac{\partial u}{\partial\nu}+\lambda u=0 &\mbox{on }\partial\Omega,
			\end{cases}
			$$
			where $\Omega\subset\R^2$ is a bounded smooth domain. We prove two existence results under a suitable relation between $\epsilon$ small and $\la$ large. When $\Omega$ is symmetric with respect to an axis, we prove the existence of a family of solutions $u_{\epsilon,\lambda}$ concentrating at two points with different spin, both located on the symmetry line and close to the boundary. In the second result, we assume $\Omega$ is not simply connected and we construct sign-changing solutions concentrating at points located close to the boundary, each of them on a different connected component of the boundary.
	\end{abstract}
	

	\maketitle

	\section{Introduction}

	Let $\Omega \subset \R^2$ a bounded domain with smooth boundary.
	In this paper we are interested in study of the singular perturbation problem
	\begin{equation}\label{eq}
		\left \{ \begin{array}{rll} \Delta u + \epsilon^2 (e^u - e^{-u}) & = 0 \quad & \mbox{in} \ \Omega, \\ 
			\RR_\lambda u := \ds \pd{u}{\nu} + \lambda u & = 0 \quad & \mbox{on} \ \partial \Omega,     
		\end{array} \right .
	\end{equation}
	where $\nu(x)$ denotes the exterior unit normal on $x \in \partial \Omega$, and $\e, \lambda$ are positive parameters. Our main concern is the construction of sign-changing solutions when $\epsilon$ is small, and simultaneously, $\lambda$ is large. 
	
	
	Robin boundary condition (also known as boundary condition of the third type) can be seen as a combination of Neumann and Dirichlet boundary condition, in a proportion cast by $\la$. In \cite{BW}, Berestycki and Wei study concentration phenomena for the least energy solution of equations of Ni-Takagi type with Robin boundary condition. In fact, they prove the existence of $\la^*$ such that,  for $\lambda \geq \la^*$, the problem resembles the behavior of the Dirichlet problem studied (c.f.~\cite{nw}), meanwhile for $\lambda < \la^*$ problem is closer to the one with Neumann boundary condition (c.f.~\cite{nt}). In terms of applications, Robin boundary condition is considered in biological models~\cite{DMO}, and thermal conductivity~\cite{clarte}.
	
	If we formally take the limit $\lambda \to \infty$ in~\eqref{eq}, the problem becomes
	\begin{eqnarray}\label{eqli}
		\left\{ \arraycolsep=1.5pt
		\begin{array}{rll}
			\Delta u+\epsilon^2 (e^u - e^{-u}) & =0,\ \ \quad &
			{\rm in}\ \Omega, \\[2mm]
			u & =0,\ \ \quad & {\rm on}\
			\partial\Omega.
		\end{array}
		\right.
	\end{eqnarray}
	
	This equation is tipically referred as \textsl{$\sinh$-Poisson equation} with Dirichlet boundary condition, and it has connection with the description of two-dimensional turbulent Euler flows, see Onsager \cite{onsager, jomo,mojo, chorin,mapu} for a physical discussion of this problem.
	Roughly speaking, the location of vortices in the flow can be described through concentration points of the solution to~\eqref{eqli}. 
	
	In this respect, Bartolucci and Pistoia~\cite{barpis} prove that for every $m \in \N$, there exists a solution to problem~\eqref{eqli} that concentrates around stable critical points of the Hamiltonian associated to the Dirichlet problem in $\Omega$, given by
	\begin{equation}\label{F}
		\varphi_m(\xi_1, ..., \xi_m) = \sum_{i = 1}^m H(\xi_i, \xi_i) + \sum_{i \neq j} a_i a_j G(\xi_i, \xi_j), \quad (\xi_1, ..., \xi_k) \in \Omega^k,
	\end{equation}
	as $\epsilon \to 0$. Here, $a_i \in \{ -1, 1 \}$ for all $i=1,...,m$ determine the spin configuration of the concentrating points; $G$ is the Green function with Dirichlet boundary condition
	\begin{eqnarray*}
		\left\{\arraycolsep=1.5pt
		\begin{array}{rll}
			-\Delta_x G(x,y)&=8\pi \delta_y(x)\ \ \ \ \quad &
			\mbox{in}\ \Omega, \\[2mm]
			G(x,y)&=0\ \ \quad
			&  \mbox{on}\ \partial\Omega,
		\end{array}
		\right.
	\end{eqnarray*}
	and $H$ is the regular part of the Green function, also referred as the associated Robin function. Moreover, it is proven in~\cite{barpis} that
	$$
	u^\epsilon(x) \to 8\pi\sum_{i=1}^k a_i G(x,\xi_i) \quad \mbox{in} \ C^{1,\sigma}(\bar \Omega \setminus \{ \xi_1,..., \xi_m \} ), \ \mbox{as} \ \epsilon \to 0,
	$$
	and the solution concentrates each $\xi_i$ with a sign determined by $a_i$. Concentration points are away to the boundary in consonance with the Dirichlet condition.
	
	Positive concentrating solutions, concentrating around critical points of a Hamiltonian function draws back to the early nineties in the seminal work of Nagasaki and Suzuki~\cite{ns} for the Liouville equation (namely, when $e^u - e^{-u}$ is replaced by $e^u$ in~\eqref{eqli}) with Dirichlet boundary condition, and it is shown the effect of the domain determines the existence of concentration, and also the location of the blow-up points. In the case of the $\sinh$-Poisson equation, negative concentration points are allowed . In~\cite{barpis}, the authors provide the existence of solutions with two concentration points with different sign, provided $\Omega$ is axially symmetric. This result was later extended by Bartsch, Pistoia and Weth
	in~\cite{BPW}, where the authors prove the existence of solutions with an arbitrary number of concentrating points, located in the symmetry axis and alternating sign. For general domains, they prove the result for $3$ and $4$ concentrating points, under suitable configuration of the spins. We also would like to mention the contributions in~\cite{deng1,dmusso} concerning construction of solutions to equations with nonlinearities of exponential type in dimension two, in~\cite{ef,f} for compact Riemann surfaces, and in~\cite{delatorre} for fractional equations.

	Concerning Robin boundary condition, in~\cite{DT} the authors address the Liouville equation with Robin boundary condition
	\begin{equation}
		\left \{ \begin{array}{rll} \Delta u + \epsilon^2 e^u & = 0 \quad & \mbox{in} \Omega, \\ \RR_\lambda u & = 0 \quad & \mbox{on} \ \partial \Omega, \end{array} \right .
	\end{equation}
	and prove the existence of positive solutions with arbitrary number of concentration points, located at critical points of the Hamiltonian function
	\begin{equation}\label{Hamiltoniana1}
		\varphi_m(\xi)=\sum_{i=1}^m H_\la(\xi_i,\xi_i) + \sum_{i\ne j}^m G_\la(\xi_i,\xi_j), \quad \xi = (\xi_1,...,m) \in \Omega^m,
	\end{equation}
	where, this time, $G_\lambda$ is the Green function
	\begin{equation}\label{Green}
		\left \{ \begin{array}{rll} -\Delta G_\lambda(x,y) & = 8 \pi \delta_y(x) \quad &  x \in \Omega \\ 
			\RR_\lambda (G_\lambda(\cdot,y)) & = 0 \quad & \mbox{on} \ \partial \Omega,          
		\end{array} \right .
	\end{equation}
	and in this case, $H_\lambda$ is the associated Robin function, defined as
	\begin{equation}\label{parteregular}
		H_\lambda(x,y) = G_\lambda(x,y) - \Gamma(x - y), \quad x, y \in \Omega, 
	\end{equation}
	where $\Gamma(x-y) = -4 \log |x - y|$ is the fundamental solution for the Laplacian in the plane. The analysis in~\cite{DT} (see also~\cite{CDF,ZS}) uses in a crucial way the asymptotic behavior $H_\lambda$ when $\lambda$ is large. Accurate asymptotics can be found in D\'avila, Kowalczyk and Montenegro in~\cite{DKM}, showing that the behavior of the map $x \mapsto H_\lambda(x,x)$ in the normal direction to the boundary develops a strict minima at distance of order $O(\lambda^{-1})$ to $\partial \Omega$. Thus, critical points for $\varphi_m$ in~\eqref{Hamiltoniana1} can be found for tuples $(\xi_1,...,\xi_m)$ such that each $\xi_i$ is sufficiently close to the boundary, and away each other. In fact, two types of different solutions can be found, one associated to minima of $\varphi_m$, and the other associated to a critical point of linking-type. 
	
\medskip
	
Our approach follows several of the concepts discussed above. In the first main result, we rely on axial symmetry just as in~\cite{barpis, BPW} to conclude the existence of sign-changing, two-point concentrating solution in a certain regime of $\epsilon$ small and $\lambda$ large. 
	\begin{teo}\label{simetria}
		Assume that $\Omega \subset \R^2$ is bounded domain with smooth boundary, and symmetric with respect to the axis $x$. Denote $(a,b) = \Omega \cap \R \times \{ 0 \}$. Then, for each $\alpha > 16$, there exist $\la_0>1, \e_0 \in (0,1)$ such that, for every $\la\geq\la_0$ and $\e$ satisfying $\e \la^{\alpha} \leq \epsilon_0$, problem \eqref{eq} has a solution $u_{\e,\la}$ that concentrates with different sign at two points $\xi_i = (t_i, 0), i=1,2$, with 
		$$
		|t_1 - a|, |t_2 - b| = O(\lambda^{-1}).
		$$
	\end{teo}
	
	\medskip
	
	Here, by \textsl{concentration at a point} $\xi \in \Omega$ we mean that for all $\delta \in (0,1)$, $\sup_{B_\delta(\xi)} |u_{\e, \la}| \to +\infty$ as $\epsilon \to 0, \la \to \infty$. Notice that by the invariance of the problem, if $u_{\epsilon, \la}$ is solution, $-u_{\epsilon, \la}$ is also a solution.

	\medskip
	
		 The proof of the previous theorem rely on Lyapunov-Schmidt reduction, and most of the technical arguments are performed over the equation~\eqref{L}, equivalent to~\eqref{eq}. Our approach in the construction of approximate solution are rather close to those present in \cite{DT,dkm,dmusso}. We use a two-parameter family of entire solutions of the Liouville equation $\R^2$ as the first ansatz, in junction with smooth correctors that allows us to satisfy the boundary condition. The reduced problem corresponds to that of adjusting variationally the location of the concentrating points, as critical points of an energy functional associated to the weak formulation of the problem. In a saturation regime of the parameters, that functional is close to the Hamiltonian $\varphi_m$, which is given by
	\begin{equation}\label{Hamiltoniana2}
		\varphi_m(\xi)=\sum_{i=1}^m H_\la(\xi_i,\xi_i) + \sum_{i\ne j}^m a_i a_jG_\la(\xi_i,\xi_j), \quad \xi = (\xi_1,...,m) \in \Omega^m,
	\end{equation}
	with $a_i \in \{ -1,1 \}$ for $i=1,..,m$ ($m=2$ in the setting of Theorem~\ref{simetria}). Notice that in contrast with~\eqref{Hamiltoniana1}, this time the contribution of the Green function maybe unbounded from below for tuples $(\xi_1, \xi_2) \in \Omega \times \Omega$ close to the diagonal if $a_1 a_2 = -1$, and this makes more difficult to locate concentrating points as minima. The relation among $\epsilon$ and $\lambda$ in the theorem allows us to control the error terms to get the criticality from the minima of the Robin function showed in~\cite{DKM}, which occurs close to boundary. In particular, the method leads to a relation among $\epsilon$ and $\lambda$ involving logarithmic corrections that explains the corresponding hypothesis in Theorem~\ref{simetria}, see Proposition~\eqref{fju}. For simplicity, we state our main result in a simpler form, at the expense of a non sharp estimate.

	Concentration near the boundary is in big contrast with previous results regarding interior concentration, see for instance~\cite{ns, bm, ls, mw, bp, emp, dkm}. We believe that by methods similar to the ones presented here, solutions with interior concentration can be constructed for~\eqref{eq}, since we expect to have $H_\lambda(\cdot, \xi) \to H(\cdot, \xi)$ as $\lambda \to \infty$, locally uniform in $\Omega$, where $H$ is the regular part of the Green function with Dirichlet boundary condition. These solution are associated to maxima for the Hamiltonian function. We do not pursue in this direction and focus on concentration near the boundary.
	
	%

	In our second main result, we are able to replace the symmetry assumption of the previous theorem by a topological one. 
	\begin{teo}\label{holes}
		Assume that $\Omega$ is not simply connected, and let $\Gamma_1, .., \Gamma_n$, $n \in \N$, be the connected components of $\partial \Omega$, and let $k \in \{1,...,m\}$. Then, for each $\alpha > 16$ there exist $\la_0>1$ and $\e_0 \in (0,1)$ such that for each $\la\geq\la_0$ and each $\e$ satisfying $\e \la^{\alpha}\leq \epsilon_0$, problem \eqref{eq} has a sign-changing solution $u_{\e,\la}$ that concentrates at points $\xi_1,...,\xi_k \in \Omega$, with $\mathrm{dist}(\xi_i, \Gamma_i) = O(\lambda^{-1})$ for all $i=1,...,k$, and where $\Gamma_i \neq \Gamma_j$ if $i \neq j$.
	\end{teo}

	This theorem follows the same strategy of the previous result, where as before the concentration points are also minima to the Hamiltonian function. In this case, the uniform distance among connected components of the boundary of the domain allows us to control the contribution of the Green function in $\varphi_m$, no matter the sign of the interacting points is. In fact, in the notation of the theorem, we have $\sum_{k=1}^{n} 2^k \binom{n}{k}$ different solutions (including the ones concentrating at one point, and the solutions due to the invariance of the equation).

	It is well-known that the use of index/linking arguments have been successful for the study of criticality of smooth functionals, and it has been employed in problems similar to ours, see for example~\cite{dkm, bp, DT} and its references. It would be interesting to know if solutions of saddle-point type could be obtained for this problem, but we did not pursue in this direction.

	\medskip
	
	%

	%

	
	\medskip
	
	This paper is organized as follows. In Section \ref{sec2}, describing a first approximation solution to problem \eqref{eq} and estimating the error. 
	Section \ref{finite} is
	devoted to perform the finite dimensional reduction. In Section \ref{nonlinear} we study the associated nonlinear problem. Section \ref{energyex} contains the asymptotic expansion of
	the reduced energy.  Finally, in Section \ref{secmain} we will prove our main results. 

	\section{Preliminaries and ansatz for solutions}\label{sec2}
	
	In this section, we denote $d(x) = \mathrm{dist}(x, \partial \Omega)$ for each $x \in \Omega$.
	
	\subsection{Preliminaries about Green and Robin function.} We start providing some estimates for the Green function with Robin boundary condition $G_\la$ in~\eqref{Green}, and its regular part $H_\la$ in~\eqref{parteregular}. 
	\begin{lema}\label{gest}
		Let $\xi\in\Om$.
		\begin{itemize}
			\item For each $\delta > 0$ small, there exists $C$ depending on $\delta$ and $\Omega$ such that if $d(\xi) \geq \delta$ we have
			$$
			\| H_\la(\cdot, \xi) \|_\infty \leq C.
			$$
			
			\item There exists $\delta \in (0,1)$ small and $C_\Omega > 0$ large depending on the smoothness of $\Omega$, such that, for all $\la$ large in terms of $\delta$, if  $(\la \log \la)^{-1} \leq d(\xi) \leq \delta$, then 
			\begin{equation}\label{Hxi}
				-C_\Omega + 2\log |x - \xi^*|\leq H_\la(x,\xi) \le C_\Omega,
			\end{equation}
			where $\xi^* \in \Omega^c$ is the reflection of $\xi$ with respect to the tangent to $\partial \Omega$ supported at the projection of $\xi$.

			\item In particular, for any $0<\delta$ there exists $C_\delta > 0$ such that 
			\begin{equation}\label{Gfxi}
				\|G_\la(\cdot,\xi)\|_{L^\infty\lf(\Om\setminus \overline{B_\delta(\xi)} \rg)}\le C_\delta.
			\end{equation}
			
		\end{itemize}
	\end{lema}
	
	\begin{proof}
		Let $\xi \in \Omega$ close to the boundary. 
		Up to rotation and translation, we can assume that the projection of $\xi$ onto the boundary is the origin, and $\nu(0) = -e_2$. We denote $\xi^* = (0, -d(\xi))$. 
		
		We divide the analysis depending on the distance of $\xi$ to the boundary. Let $\tilde{\delta} \in (0,1)$ to be specified later on. 
		
		By definition of $H_\lambda$, for each $x \in \partial \Omega$ we have 
		\begin{align*}
			\RR_\lambda H_\lambda(x) = \RR_\lambda (-\Gamma) = & |x - \xi|^{-2} (x -\xi) \cdot \nu(x) + \lambda \log|x - \xi|.
		\end{align*}

		If $d(\xi) \geq \tilde{\delta}$, then it is easy to see using maximum principle for Robin boundary condition (see Lemma 2.1 in~\cite{DT}) that 
		$$
		\| H_\lambda (\cdot, \xi) \|_\infty \leq C_{\tilde{\delta}}.
		$$
		
		From here, we assume $(\lambda \log \la)^{-1} \leq d(\xi) \leq \tilde{\delta}$.
		
		For $x \in \partial \Omega$ such that $|x - \xi| \geq \tilde{\delta}$ we have 
		$$
		|\RR_\lambda H_\lambda(x)| \leq \la C_\Omega,
		$$
		where $C_\Omega = \log \max \{ 1, \mathrm{diam}(\Omega)\} + \tilde{\delta}^{-1}$.
		
		From here, we concentrate on the points $x$ on the boundary such that $|x - \xi| \leq \tilde{\delta}$. 
		
		We start with upper bounds for $\RR_\la H_\la$. If $|x - \xi| \leq \la^{-1}$, then
		$$
		\RR_\la H_\la(x) \leq \la \log \la + \la \log |x - \xi| \leq 0,
		$$
		meanwhile, for $\la^{-1} \leq |x - \xi| \leq \tilde{\delta}$ we have 
		$$
		\RR_\la H_\la(x) \leq \la + \la \log(\tilde{\delta}) \leq 0,
		$$
		provided $\tilde{\delta} \leq e^{-1}$.
		
		From here, using maximum principle for Robin boundary condition, using $C_\Omega$ as supersolution, we conclude the upper bound in~\eqref{Hxi}.

		\medskip
		
		For the lower bound, we start recalling that for all $|x - \xi| \leq \tilde{\delta}$ on the boundary, we have $|x - \xi|$ and $|x - \xi^*|$ are comparable, in the sense that there exists $c_\Omega \in (0,1)$ such that
		$$
		c_\Omega \leq \frac{|x - \xi|}{|x - \xi^*|} \leq c_\Omega^{-1}, \quad x \in \partial \Omega \cap B(\xi, 3\tilde{\delta}).
		$$
		
		This is a consequence of the smoothness of the domain. 
		
		If $|x - \xi| \geq \la^{-\theta}$ for $\theta \in (1/2, 1)$, we have
		\begin{align*}
			\RR_\la H_\la \geq & -|x - \xi|^{-1} + \la \log|x - \xi|\\
			\geq &  -\la^\theta + \la \log |x - \xi^*| + \la \log c_\Omega \\
			\geq & 2 \la (\log |x - \xi^*| + \log c_\Omega).
		\end{align*}
		
		Tomamos $S = \log |x - \xi^*|$ y 
		\begin{align*}
			\RR_\la S = \widehat{x - \xi^*} \cdot \nu(x) |x - \xi^*|^{-1} + \la \log |x - \xi^*| & \leq c_\Omega^{-\theta} \la^{\theta} + \la \log |x - \xi^*| \\
			& \leq \lambda (c_\Omega^{-\theta} \la^{\theta - 1} + \log |x - \xi^*|) \\
			\leq & \frac{1}{2} \la \log |x - \xi^*|
		\end{align*}
		
		If $|x - \xi| \leq \la^{-\theta}$, we use that $x$ can be writen as $x = (x', \psi(x'))$ with $\psi(0) = 0, \psi'(0) = 0$, $\psi''$ bounded, from which
		\begin{equation*}
			\widehat{(x - \xi)} \cdot \nu(x) = \frac{x'\psi' - \psi + d(\xi)}{\sqrt{1 + (\psi')^2} \sqrt{|x'|^2 + (\psi - d(\xi))^2}} \geq \frac{1}{2},
		\end{equation*}
		for all $\la$ large enough. This implies that
		\begin{equation*}
			\RR_\la H_\la \geq \frac{c_\Omega}{2|x - \xi^*|} + \la \log |x - \xi^*| + \la \log c_\Omega.
		\end{equation*}
		
		Thus, we use the function $S(x) = 2\log |x - \xi^*| - 2\log c_\Omega$, which, in view of the estimates above, satisfies $\RR_\la H_\la \geq \RR_\la S$ on $\partial \Omega$. Hence, we use maximum principle again, to conclude that $H_\la \geq S$ in $\Omega$. This leads to the lower bound in~\eqref{Hxi}. The remaining estimates can be easily obtained from~\eqref{Hxi} and the definition of $H, G$ and $\Gamma$. This completes the proof.  
	\end{proof}

	\subsection{Ansatz for the solution. Reduced problem.}

	In this section we will construct an approximation of the solution to problem \eqref{eq}. Then we estimate the error of such approximation in a suitable norm. The basic idea is to consider a parameter $\mu > 0$ and the functions
	\begin{equation}\label{basicbubble}
		w_\mu(x) = \log\frac{8\mu^2}{(\mu^2 + |x|^2)^2},\quad x\in\R^2
	\end{equation}
	which are solutions to the Liouville equation in the whole plane
	\begin{equation}\label{Liouville}
		\Delta u + e^u = 0 \quad \mbox{in} \ \R^2.
	\end{equation}
	

	An actual solution will have an asymptotic profile as $\epsilon \to 0$ and $\lambda \to +\infty$ which resembles these solutions, properly translated and rescaled in terms of these parameters. Specifically, we choose our scaling parameter as
	\begin{equation}\label{rho}
		\rho := \frac{\epsilon}{\lambda^2}.
	\end{equation} 
	
	Given $m \in \N$, we consider $\{ \mu_j \}_{j=1}^m \subset (0,+\infty)$,  $\{ \xi_j \}_{j=1}^m \subset \Omega$, and for each $j=1,...,m$, we denote
	\begin{equation}\label{defwj}
		\begin{split}
			w_j(x) & =  w_{\mu_j \rho}(x - \xi_j) + 2\log \frac{1}{\epsilon}  =  \log  \frac{8\mu_j^2}{(\mu_j^2 \rho^2 + |x - \xi_j|^2)^2}  + 2\log \frac{\rho}{\epsilon}, 
		\end{split}
	\end{equation}
	for all $x \in \R^2$. It is easy to see that for each $j$ the function $w_j$ satisfies
	\begin{align}\label{eqwj}
		\Delta w_j + \epsilon^2 e^{w_j} = 0 \quad \mbox{in} \ \R^2.
	\end{align} 
	

	In order to satisfy the Robin boundary condition we introduce harmonic functions $H_j$ satisfying 
	\begin{equation}\label{E}
		\left \{ \begin{array}{rll} -\Delta H_j & = 0 \quad & \mbox{in} \ \Omega, \\ \RR_\lambda (H_j) & = - \RR_\lambda(w_{j}) \quad & \mbox{on} \ \partial \Omega. \end{array} \right .
	\end{equation}
	
	Our single-point ansatz takes the form 
	\begin{equation}\label{Uj}
		U_j(x) = w_j(x) + H_j(x), \quad x \in \Omega.
	\end{equation}

	Using the explicit form of $w_j$ in \eqref{defwj} together with \eqref{E}, we can use maximum principle in Lemma 2.1 in~\cite{DT} over $x \mapsto H_\la(x, \xi_j)-H_j(x)$ to conclude that 
	\begin{equation*}
		H_j(x) + 2\log\frac{\rho}{\epsilon} + \log(8\mu_j^2) - H_\la(x, \xi_j) = O\bigg(\frac{\mu_j^2 \rho^2}{d_j^2}\bigg) +  O\left(\frac{\mu_j^2 \rho^2}{\lambda d_j^3}\right),
	\end{equation*}
	uniformly in $\bar \Omega$. Here and in what follows, we have adopted the notation $d(x) = \mathrm{dist}(x, \partial \Omega)$ for $x \in \Omega$, and $d_j = d(\xi_j)$. 
	
	Thus, in view of the definition of $w_j$ and the estimate above, we have the expansion
	$$
	U_{j}(x) = G_\lambda(x, \xi_j) -2\log \Big(1+\frac{\mu_j^2\rho^2}{|x-\xi_j|^2}\Big)+ O\Big(\frac{\mu_j^2 \rho^2}{d_j^2}\Big).
	$$
	
	Now we provide the first estimates concerning the ansantz when we consider the expected location of the concentration points $\{ \xi_j \}$.
	\begin{lema}\label{1}
		Let $ \{\xi_j \}_{j=1}^m \subset \Omega, \{ \mu_j \}_{j=1}^m \subset (0,+\infty)$ satisfying 
		\begin{align}
			\label{setm} & d_j \in (\delta \lambda^{-1}, \delta^{-1} \lambda^{-1}), \\
			\label{muacotadas} & \mu_j \in (\delta, \delta^{-1}),
		\end{align}
		for some $\delta \in (0,1)$. Then, for each $j=1,...,m$ we have
		\begin{equation}\label{estimacion1}
			\begin{split}
				H_j(x) = H_\la(x, \xi_j) - \log(8\mu_j^2) +4\log \la + O\big(\rho^2 \la^{2} \big),
			\end{split}
		\end{equation}
		uniformly in $\bar\Om$; and for each $K$ compact subset of $\bar\Om\sm \{\xi_j\}$, there exists $C_K > 0$ such that
		\begin{equation}\label{UG}
			|U_{j}(x) - G_\lambda(x, \xi_j)| \leq C_K \rho^2 + O(\la^2 \rho^2) \quad \mbox{in} \ K,
		\end{equation}
		where the $O$ term is independent of $K$.
	\end{lema}


	Consider $\{ a_j \}_{j=1}^m$ with $a_j \in \{-1,1\}$ for all $j$. We introduce the first approximation of the problem~\eqref{eq} as
	\begin{align}\label{defU}
		U(x) := \sum_{j=1}^{m} a_j U_j(x).
	\end{align}
	

	Following the directions of~\cite{dkm}, we study the problem in expanded variables depending on $\rho$ given in~\eqref{rho}. For this, we consider the function $v(y)=u(\rho y)$ for $y\in\Omega_\rho:=\{\xi\in\R^2:\,\rho\xi\in\Omega\}$, so that $u$ is a solution to \eqref{eq} if and only if $v$ satisfy
	\begin{equation}\label{line}
		\left \{ \begin{array}{rll} \Delta v+\rho^2\epsilon^2 (e^v - e^{-v}) & =0  \quad &\mbox{in }\Omega_\rho, \\
			\RR_{\la \rho}  v & =0 \quad &\mbox{on }\partial\Omega_\rho. \end{array} \right .
	\end{equation}
	
	We also adopt the notation $\xi_j'=\frac{1}{\rho}\xi_j$, $V_j(y)=U_j(\rho y)$ and $V(y)=U(\rho y)$ for $y \in \Omega_\rho$ and we look for solutions to \eqref{line} in the form $v=V+\phi$ with $\phi: \Omega_\rho \to \R$ small in an adequate norm. 
	
	Thus, considering $v = V + \phi$, problem~\eqref{line} can be equivalently formulated as
	\begin{equation}\label{line1}
		\left \{ \begin{array}{rll} L \phi  & =-[ R+\Lambda \phi  + N(\phi)] \quad & \mbox{in }\Omega_\rho, \\
			\RR_{\la \rho} \phi & = 0 \quad  &\mbox{on }\partial\Omega_\rho, \end{array} \right .
	\end{equation}
	where
	\begin{align}
		\label{L} & L \phi =\Delta\phi+W\phi, \quad\mbox{ with }\quad W=\sum_{j=1}^m \frac{8\mu_j^2}{(\mu_j^2 + |y-\xi_j'|^2)^2 }, \\
		\label{lam} & \Lambda \phi = \Big[(\epsilon\rho)^2(e^V+e^{-V}) - W\Big]\phi, \\
		\label{N} & N(\phi) =(\epsilon\rho)^2\lf[e^{V}(e^{\phi}-\phi-1)-e^{-V}(e^{-\phi}+\phi -1)\rg], \\
		\label{R} & R(y) = \Delta V+(\epsilon\rho)^2(e^V-e^{-V}). 
	\end{align}

	
	\subsection{First estimates for problem~\eqref{line1}.} In what follows, we recall relevant estimates concerning the Robin function $H_\la$ we collect from~\cite{DKM} for the two-dimensional case. Define
	\begin{align}
		\label{expanrobin1} & h(\theta)= -4\log(2\theta)+8\int_0^\infty e^{-t}\log\left(2\theta+t\right)dt, \\
		\label{expanrobin2} & v(\theta)=-2\theta-4\theta\int_{0}^\infty \frac{e^{-2\theta s}}{(1+s)^2}ds.
	\end{align}
	
	It is possible to see that $h: (0,\infty)\to\mathbb{R}$ has a unique nodegenerate minimum $\theta_{0}\in(0,\infty)$.
	
	The Robin function obeys the expansion
	\begin{equation}\label{expanrobin}
		H_\la(x,x)=-4\log\la + h(\la d(x))+\la^{-1}\kappa(\tilde{x})v(\la d(x))+O(\la^{-1-\alpha})
	\end{equation}
	for each $x$ such that $a_1\leq \la d(x) \leq a_2$ for some constants $0<a_1<a_2$, and all $\la>0$ large enough. Here, $\alpha \in (0,1)$, $\tilde x$ is the projection of $x$ onto the boundary, 
	$\kappa(\tilde{x})$ is the mean curvature of $\partial\Omega$ at
	$\tilde{x}$, see Lemma 2.1 in~\cite{DKM}.

	\medskip
	
	For reasons that will be made clear later in the next section, for measurable functions $h: \Omega_\rho \to \R$ we introduce the following weighted norm
	\begin{equation}\label{normstar}
		\|h\|_*=\sup_{y\in \Om_\rho} \lf(\sum_{j=1}^m(1+|y-\xi_j'|)^{-2-\sigma} +\rho^2\rg)^{-1}|h(y)|,
	\end{equation}
	for $0<\sigma<1$ fixed.
	
	\begin{lema}\label{rest}
		Let $\{ \xi_j\}_{j=1}^m \subset \Omega$ such that~\eqref{setm} holds for some $\delta > 0$, and assume further that
		\begin{equation}\label{xiseparados}
			|\xi_i - \xi_j| \geq \delta \quad \mbox{for} \ i \neq j.
		\end{equation}
		
		Set $\{ \mu_j \}_{j=1}^m$ as
		\begin{align}\label{condmu}
			\log(8 \mu_j^2) = H_\la(\xi_j, \xi_j) + \sum_{i \neq j} a_i  G_\la(\xi_i, \xi_j) + 4\log \lambda.
		\end{align}
		
		Then, there exists $C > 1$ just depending on $\delta$ such that $C^{-1} \leq \mu_j \leq C$  for all $j=1,..,m$. 
		
		Moreover, we have
		\begin{equation}\label{estR}
			\|R\|_*\le C\rho\la^9\log\la =C \epsilon\lambda^7\log\la,
		\end{equation}
		and 
		$$
		\|\Lambda\phi \|_*\le C\rho\la^9\log\la \|\phi\|_{L^\infty (\Om_\rho)},
		$$
		for all $\phi \in L^\infty(\Omega_\rho)$.
	\end{lema}
	
	\medskip
	\noindent
	{\bf \textit{Proof:}} The uniform estimates for $\mu_j$ are a consequence of the expansion of the Robin function~\eqref{expanrobin}, together with the estimates for the Green function in Lemma~\ref{1}.
	
	We estimate $R$ by dividing the analysis in different regions. Note that by definition of $V_i$ and \eqref{eqwj} we see that for any $i=1,\dots,m$
	\begin{equation*}
		\Delta V_i(y) = \rho^2\Delta w_i(\rho y) = - \rho^2\epsilon^2  e^{w_i(\rho y)} = -\frac{8\mu_i^2 \rho^4}{(\mu_i^2 \rho^2 + |\rho y - \xi_i|^2)^2}.
	\end{equation*}
	
	\medskip
	
	\noindent
	\textsl{Case 1:} Assume that $|y - \xi_j'| \geq \dfrac{\delta}{\rho}$ for all $j = 1,...,m$ with $\delta > 0$ small but independent of $\epsilon, \lambda$. Then, for each $j$ and for $y \in \Omega_\rho$ with $|y - \xi_j'| \geq \dfrac\delta\rho$ we have from  \eqref{eqwj} that
	\begin{equation}\label{DeltaU}
		\Delta V_j(y) =  - \frac{8\mu_j^2 \rho^4}{(\mu_j^2 \rho^2 + |\rho y - \xi_j|^2)^2} = O(\rho^4),
	\end{equation} 
	where the $O$ term is uniform if $\delta > 0$ and $\mu_j > 0$ are fixed. On the other hand, using the estimate~\eqref{UG} we have
	\begin{equation*}
		\max_{a = \pm 1} \{ e^{aV_j(y)} \} \leq C \max_{a = \pm 1} \{ e^{a G_\la(\rho y, \xi_j)} \}. 
	\end{equation*}
	From Lemma \ref{gest} we deduce that $e^{V_j}$ is uniformly bounded away $\xi_j$. This together with~\eqref{DeltaU} leads us to
	\begin{equation}\label{R1}
		|R(y)| = |\Delta V + \rho^2\epsilon^2 g(V)| \leq C(\rho^4 + \rho^2\epsilon^2) \quad \mbox{ in } \ \Omega_\rho \setminus \bigcup_{j=1}^m \overline{B_{\delta/\rho} (\xi_j')} .
	\end{equation}
	
	\medskip
	\noindent
	\textsl{Case 2:} Now, assume that $\delta/(\rho\lambda\log\la ) \leq |y - \xi_j'| \leq \delta/\rho$ for some $j \in \{ 1,\dots ,m\}$. As in \eqref{DeltaU} in the previous case, we see that
	\begin{align*}
		|\Delta V_j(y)|= \frac{8\mu_j^2 \rho^4}{(\mu_j^2 \rho^2 + \delta^2[\lambda\log\la]^{-2})^2} = O(\rho^4 \lambda^4\log^4\la) = O(\rho^2\epsilon^2\log\la).
	\end{align*}
	
	Using again estimates~\eqref{UG},~\eqref{Hxi} and the fact that $d_j = O(\la^{-1})$, we arrive at
	\begin{equation*}
		e^{V(y)}=e^{a_jG_\la(\rho y,\xi_j)+\sum_{l\ne j} a_lG_\la(\rho y,\xi_l) + O(\rho^2\la^2\log^2\la)}.
	\end{equation*} 
	
	If $a_j=1$ then we get that
	$$e^{V(y)}=\frac{e^{H_\la(\rho y,\xi_j)+O(1)}}{|\rho y-\xi_j|^4} =O\lf(\frac{\la^4}{\rho^4|y-\xi_j'|^4}\rg)$$
	and if $a_j=-1$ then we find that
	$$e^{V(y)}= |\rho y-\xi_j|^4e^{-H_\la(\rho y,\xi_j)+O(1)} =O\lf( \rho^4\la^4|y-\xi_j'|^4\rg).$$
	Taking into account that $a_j=1$ becomes $-1$ in $e^{-V(y)}$ and $a_j=-1$ becomes $1$ in $e^{-V(y)}$, a similar argument to estimate $e^{-V(y)}$ lead us to obtain
	\begin{align*}
		|\rho^2\epsilon^2 g(V	(y))| \leq C \rho^2\epsilon^2\lf[\frac{\la^4}{\rho^4|y-\xi_j'|^4}+ \rho^4\la^4|y-\xi_j'|^4\rg],
	\end{align*}
	for $y\in B_{\delta/\rho}(\xi_j') \setminus B_{\delta/(\rho\lambda\log\la)}(\xi_j')$ and for some constant just depending on $\delta$. This estimate allows us to conclude that
	\begin{equation}\label{R2}
		|R(y)| \leq C  \epsilon \lambda^7\log\la \sum_{j=1}^{m} \frac{1}{1 + |y - \xi_j'|^3}, \quad y \in \Omega_\rho \cap (B_{\delta/\rho}(\xi_j') \setminus B_{\delta/(\lambda \rho\log\la)}(\xi_j')),
	\end{equation}
	in view of
	$$\rho^2 \e^2\lf[\log\lambda+  \frac{\la^4}{\rho^4|y-\xi_j'|^4}+ \rho^4\la^4|y-\xi_j'|^4\rg][1+|y-\xi_j'|^3]\le C\e\la^7\log\la $$
	for $y \in \Omega_\rho \cap (B_{\delta/\rho}(\xi_j') \setminus B_{\delta/(\lambda \rho\log\la )}(\xi_j'))$.
	
	\medskip
	\noindent
	\textsl{Case 3:} Finally, assume that $|y - \xi_j'| \leq \delta/(\rho\lambda\log\la)$ for some $j\in\{1,\dots,m\}$. In this case, by the condition over the points $\xi_j$ we have that $|\xi_j - \xi_i| \geq \delta$ when $i \neq j$ (we can choose $\tilde{\delta}>\de$ if necessary) and therefore we can use the estimates of Case 1 for the portion of $R$ relative to points indexed by $i \neq j$.

	Note that we get that for $j$
	\begin{equation}\label{hola}
		\Delta V(y) = - a_j \frac{8\mu_j^2 }{(\mu_j^2  + | y - \xi_j'|^2)^2} + O(\rho^4),
	\end{equation}
	in the considered region where the $O$-term depends only on $\delta$ and $\mu_j$.

	Now, we deal with the exponential term. We first assume that $a_j = 1$. Notice that
	\begin{equation*}
		\begin{split}
			\rho^2\epsilon^2 g(V) 
			&= \rho^2\epsilon^2 (e^{w_j(\rho y)} e^{H_j(\rho y) + \sum_{i \neq j} a_i V_i} - e^{-V} ).
		\end{split}
	\end{equation*}
	By estimates~\eqref{estimacion1} we see that
	\begin{align*}
		e^{V(y)} &=  e^{w_j(\rho y)} e^{H_j(\rho y) + \sum_{i \neq j} a_i V_i}  \\
		&= e^{w_j(\rho y)}\exp \Big{(} H_\la(\rho y, \xi_j) + \sum_{i \neq j} a_i G_\la (\rho y, \xi_i) - \log(8\mu_j^2) \\
		&\qquad\qquad + 4\log\lambda + O\Big(\frac{\mu^2_j \rho^2}{d_j^2}\Big) \Big{)}
	\end{align*} 
	
	At this point, we notice that by a first-order Taylor expansion, for $y$ in this region we have
	\begin{equation}\label{tehxi}
		H_\la(x,\xi_j)=H_\la(\xi_j,\xi_j) + O(\la\log \lambda \, |x-\xi_j|), 
	\end{equation}
	and for $i\ne j$ we have
	\begin{equation}\label{tegxi}
		G_\la(x,\xi_i)= G_\la(\xi_j,\xi_i) + O(\lambda\log\lambda \, |x-\xi_j|),
	\end{equation}
	where the $O$ terms are inependent of $\epsilon$ and $\la$.
	
	Using this and assumption~\eqref{condmu}, we arrive at
	\begin{align*}
		\rho^2\epsilon^2 e^{V(y)} & = \rho^2 \epsilon^2 e^{w_j(\rho y)}\cdot \exp \bigg{(} O(\lambda\log\la \, |\rho y - \xi_j|)  + O\Big(\frac{\mu^2_j \rho^2}{d_j^2}\Big) \bigg{)} \\
		&= \frac{8\mu_j^2}{(\mu_j^2 + |y - \xi_j'|^2)^2}\lf[1 + O(\rho \lambda\log\la |y - \xi_j'|) +  O(\rho^2 \la^2) \rg] 
	\end{align*}
	Using the same computation, it is possible to see that
	\begin{align*}
		\rho^2\epsilon^2 e^{- V(y)}&=  \rho^2\epsilon^2 \frac{(\mu_j^2\rho^2 + |\rho y - \xi_j|^2)^2\epsilon^2}{8\mu_j^2\rho^2 }\; O(1)\\
		& =O\lf(\rho^4\epsilon^4[1+|y-\xi_j'|]^4\rg)=O\lf(\rho^4\epsilon^4\Big[\frac{\delta}{\rho\la}\Big]^4\rg)=O(\rho^2\epsilon^2) 
	\end{align*}
	From here, we conclude that
	\begin{equation}\label{eesv}
		\rho^2\epsilon^2 g(V) = \frac{8\mu_j^2 }{(\mu_j^2 + |y - \xi_j'|^2)^2}\left[1 + O(\rho \lambda \log\la |y - \xi_j'|) +  O(\rho^2\la^2)\right] + O(\rho^2\epsilon^2)
	\end{equation}
	Then, joining the above estimates and~\eqref{hola} we get for $y \in B_{\delta/(\rho \lambda\log\la)}(\xi_j')$
	\begin{equation}\label{erxi}
		R(y)=\frac{8\mu_j^2 }{(\mu_j^2  + |y - \xi_j'|^2)^2}\left[O(\rho \lambda\log\la |y - \xi_j'|) +  O(\rho^2\la^2)\right] + O(\rho^4+\rho^2\epsilon^2).
	\end{equation}
	We finish the proof by assuming $a_j = -1$. Similarly as above, we have the following estimates
	\begin{equation*}
		\rho^2\epsilon^2 g(V) = \rho^2\epsilon^2 (e^{V} - e^{w_j(\rho y)} e^{H_j(\rho y) + \sum_{i \neq j} a_i V_i}  ).
	\end{equation*}
	$$\rho^2\epsilon^2  e^{-V(y)}=  \frac{8\mu_j^2}{(\mu_j^2 + |y - \xi_j'|^2)^2}\lf(1 + O(\rho \lambda \log\la|y - \xi_j'|) +  O(\rho^2 \la^2) \rg).$$
	and
	\begin{align*}
		\rho^2\epsilon^2 e^{V(y)}& =O\lf(\rho^4\epsilon^4[1+|y-\xi_j'|]^4\rg)=O(\rho^2\epsilon^2).
	\end{align*}
	Thus, we get \eqref{eesv} and using \eqref{hola} we also obtain \eqref{erxi}. 
	From the choice of $\rho$, the definition of $\| \cdot \|_*$ and estimates \eqref{R1}, \eqref{R2} and \eqref{erxi} we deduce \eqref{estR}. The estimate for $\Lambda \phi$ follows the same computations. This completes the proof.\qed
	
	\medskip
	%
	%

	\section{Linear problem}\label{finite}
	
	In this section we shall reduce the solvability of \eqref{line1} by using the so-called Lyapunov-Schmidt finite dimensional variational reduction and prove the main result. In this procedure an important step is the solvability theory for the linear operator, obtained as the linearization of \eqref{line1} at the approximating solution $V$, namely, \eqref{L}. Observe that, as $\e\to 0$ and $\la\to+\infty$, formally the operator $L$, around a point $\xi_j'$, approaches $L_j$ defined in $\R^2$ as
	$$L_j(\phi)=\Delta\phi + \frac{8\mu_j^2}{
		(\mu_j^2+|y-\xi_j'|^2)^2} \phi,\quad j=1,\dots,m.$$

	We start this section with some notation. Given  $v_j = v_{\mu_j}$ the entire functions solving the equation
	$$
	\Delta v + \frac{8\mu_j^2}{(\mu_j^2 + |x|^2)^2} v = 0 \quad x \in \R^2,
	$$
	as a consequence of the invariance under dilations and traslations of the problem~\eqref{Liouville}, it is possible to find nontrivial solutions of this equations and, for each $x \in \R^2$, are denoted by
	\begin{equation}\label{zij}
		\begin{split}
			z_{ij}(x) = & \frac{\partial}{\partial \zeta_i} v_{j}(|x + \zeta|) \Big{|}_{\zeta = 0} = \frac{4\mu_j x_i}{\mu_j^2+|x|^2}, \quad i = 1,2, \ j = 1,...,m, \\
			z_{0j}(x) = & \frac{\partial}{\partial s} (v_{j}(|sx|) + 2\log(s)) \Big{|}_{s=1} = \frac{\mu_j^2-|x|^2}{\mu_j^2+|x|^2}, \quad j = 1,...,m.
		\end{split}
	\end{equation}
	
	The kernel of $L_j$ in $L^\infty(\R^2)$ is non-empty and is spanned by the functions $Z_{ij}$, $i=0,1,2$, due to the intrinsic invariances of the problem \eqref{line1}, where 
	\begin{eqnarray}\label{Zij}
		\ \ \ \ Z_{0j}(y) =\frac{\mu_j^2-|y-\xi_j'|^2}{
			\mu_j^2+|y-\xi_j'|^2},\quad \ \ \ Z_{ij}(y) = \frac{4\mu_j(y-\xi_j')_i}{ \mu_j^2+|y-\xi_j'|^2},\ \ i=1,2,
	\end{eqnarray}
	see \cite{bp} for a proof. Consider a large but fixed number $R_0>0$ and a radial
	smooth cut-off function $\chi$ with $\chi(r)=1$ if $r<R_0$ and
	$\chi(r)=0$ if $r>R_0+1$. Write
	\begin{eqnarray}\label{cof}
		\chi_j(y)=\chi\left(|y-\xi'_j|\right).
	\end{eqnarray}
	
	%
	%
	%
	
	Recalling $L$ in~\eqref{L}, the main result of this section is the following 
	\begin{prop}\label{inliop}
		There exist $\la_0>0$ and $\e_0>0$ such that for $\la\ge \la_0$ and $\e>0$ satisfying $0<\rho\la<\e_0$, for any points $\xi=(\xi_1,\dots, \xi_m)$ satisfying \eqref{setm} and for any $h\in L^\infty(\Om_\rho)$ with $\|h\|_*<+\infty$, there is a unique solution $\phi:=T_\la(h)$ and coefficients $c_{ij}(\xi)\in\R$, $i=1,2$, $j=1,\dots,m$ of problem
		\begin{eqnarray}\label{eq:3.5}
			\left\{ \arraycolsep=1.5pt
			\begin{array}{lll}
				\ds L(\phi)=h+ \sum\limits_{i=1}^{2}\sum_{j=1}^m c_{ij}\chi_jZ_{ij}
				\ \ &
				{\rm in}\ \Omega_\rho\\[0.4cm]
				\RR_{\rho\lambda}\phi=0\ \ & {\rm on}\ \partial\Omega_\rho \\[0.4cm]
				\displaystyle\int_{\Omega_\rho } \chi_j Z_{ij}\phi=0\ \ &  {\rm
					for}\ i=1,2,\ j=1,\dots,m.
			\end{array}
			\right.
		\end{eqnarray}
		Moreover, the map $\xi\mapsto \phi(\xi)$ is differentiable with
		\begin{equation}\label{estphicij}
			\|\phi\|_{L^\infty(\Om_\rho) }\le C|\log(\rho\la)|\ \|h\|_*, \quad |c_{ij}|\le C\|h\|_*, 
		\end{equation}
		and
		\begin{equation}\label{estdphi}
			\|\partial_{(\xi_l)_k}\phi\|_{L^\infty(\Om_\rho) }\le C|\log(\rho\la)|^2\ \|h\|_*, \quad \text{for }\ \ l=1\dots,m,\ k=1,2. 
		\end{equation}
	\end{prop}
	
	\medskip
	
	In order to prove the above result, we require some a priori estimates. We start with the following result which can be found in Lemma 3.2 in~\cite{DT}.
	\begin{lema}\label{lema1}
		There exist $\lambda_0>0$ and $\epsilon_0>0$ such that for
		$\lambda\geq\la_0$ and $\epsilon>0$ satisfying
		$0<\rho\lambda<\e_0$, any family of points $\xi =(\xi_1,\dots,\xi_m)$ satisfying \eqref{setm}
		and for any solution $\phi$ of the problem
		\begin{eqnarray}\label{eqlin}
			\left\{ \arraycolsep=1.5pt
			\begin{array}{lll}
				\Delta \phi+W\phi=h
				\ \ &
				{\rm in}\ \Om_\rho ;\\[2mm]
				\dfrac{\partial\phi}{\partial \nu}+\rho\lambda\phi= g\ \ & {\rm on}\ \partial\Om_\rho ;\\[3mm]
			\end{array}
			\right.
		\end{eqnarray}
		satisfying the orthongonality conditions
		\begin{equation} \label{orto}
			\displaystyle\int_{\Om_\rho }\phi Z_{ij}\chi_j=0 \ \ \ \
			{\rm for}\ i=0,1,2,\ j=1,\dots,m,
		\end{equation}
		we have the estimate
		\begin{equation}\label{lluvia}
			\|\phi\|_{L^{\infty}(\Om_\rho )}\leq C(\|h\|_{\ast} + \frac{1}{\lambda \rho} \| g \|_{L^\infty(\partial \Omega_\rho)}) .
		\end{equation}
	\end{lema}

	\medskip
	
	The next step is to find an a priori estimate for the solution
	avoiding the elements of the kernel due to dilations
	$\chi_jZ_{0j}$'s. 
	
	\begin{lema}\label{lema2}
		There exist $\lambda_0>0$ and $\epsilon_0>0$ such that for
		$\lambda\geq\la_0$ and $\epsilon>0$ satisfying
		$0<\rho\lambda<\e_0$, any family of points $\xi =(\xi_1,\dots,\xi_m)$ satisfying \eqref{setm}
		and for any solution $\phi$ of~\eqref{eqlin} with  $g = 0$, satisfying the orthogonality conditions
		$$
		\displaystyle\int_{\Om_\rho }\phi Z_{ij}\chi_j=0 \ \ \ \
		{\rm for}\ i=1,2,\ j=1,\dots,m,
		$$
		we have
		$$\|\phi\|_{L^{\infty}(\Om_\rho )}\leq C|\log(\rho\la)|\,\|h\|_{\ast}.
		$$
	\end{lema}
	
	\begin{proof}
		
		We will construct functions $\tilde z_j$ and fix constants $b_j \in \R$ such that the function
		$$
		\tilde \phi = \phi + \sum_{j=1}^m b_j \tilde z_j,
		$$
		satisfies, for certain $b_j \in \R$, the orthogonality conditions with respect to the dilations, and subsequenly apply Lemma~\ref{lema1}.

		We divide the proof in several steps:

		\medskip
		
		\noindent
		\textsl{1.- Construction of the correctors $\tilde z_j$:} Here we require certain estimates concerning the Green function with homogeneous Robin boundary condition in the upper half-space, that is
		\begin{equation*}
			-\Delta G(\cdot, y) = \delta_{y} \quad \mbox{in} \ \R^2_+; \qquad \RR_{a} G = 0 \quad \mbox{on} \ \{  x_2 = 0 \},
		\end{equation*}
		where $a > 0$, and in this case $\nu = -e_2$ is the exterior unit normal. As it can be seen in p. $120$ in~\cite{GT}, we have the expression
		\begin{equation}\label{G}
			G_a(x,y) = \Gamma(x - y) - \Gamma(x - y^*) - 2 \int_{0}^{+\infty} \frac{e^{-a s} (x_2 + s + y_2)}{|x + s e_2- y^*|^2} ds,
		\end{equation}
		and where for $y = (y_1, y_2) \in \R^2_+$, we denote $y^* = (y_1, -y_2)$.
		
		Let $F_j: B_r(\xi_j) \cap \Omega \to \R_+^2$ be a conformal mapping such that if we denote $\hat \xi_j \in \partial \Omega$ is the projection of $\xi_j$ onto the boundary, then $F_j(\hat \xi_j) = 0$. In addition, $DF_j(\hat \xi_j)$ is a rotation (making $DF_j(\hat \xi_j) \nu(\hat \xi_j) = -e_2$), and is such that $F_j (B_r(\xi_j) \cap \partial \Omega)$ is an interval in $\{ x : x_2 = 0 \}$. Thus,  we consider its expanded version
		$$
		F_{j, \rho}(y) = \rho^{-1} F_j(\rho y), \quad y \in B_{r/\rho}(\xi_j') \cap \Omega_\rho.
		$$ 
		
		Since $F_j$ is a conformal mapping, we have the existence of a constant $c \in (0,1)$ (just depending on $\Omega$) such that 
		\begin{equation}\label{conforme}
			c |z_1 - z_2| \leq |F_{j, \rho}(z_1) - F_{j, \rho}(z_2)|  \leq c^{-1} |z_1 - z_2|,
		\end{equation}
		for all  $z_1, z_2 \in B_{r/\rho}(\xi_j') \cap \Omega_\rho$.
		
		Taking into account that $\mathrm{dist}(\xi_j, \partial \Omega) = O(\lambda^{-1})$, we have 
		\begin{equation}\label{cotaF'}
			F_{j,\rho}(\xi_j') = \rho^{-1} d_j e_2 + O(\rho^{-1} \lambda^{-2}).
		\end{equation}

		Let $\eta_{2j}$ be a cut-off function making $\eta_{2j} = 1$ in $B_{\frac{r}{2\rho}}(\xi_j')$, $\eta_{2j} = 0$ in $B_{\frac{r}{\rho}}^c(\xi_j')$, $|D\eta_{2j}| \leq C \rho$, $|D^2 \eta_{2j}| \leq C \rho^2$ in $\R^2$, and such that $\frac{\partial \eta_{2j}}{\partial \nu} = 0$ in $\partial \Omega_\rho$. This can be done using the conformal map $F_j$ and an adequate scaling.
		
		Finally, we consider a cut-off function $\eta_{1j}$ such that $\eta_{1j} = 1$ in $B_R(\xi_j')$, $\eta_{1j} = 0$ in $B_{R + 1}(\xi_j')$ and uniform bounds for its first and second-order derivatives.
		
		For $y \in \Omega_\rho$, we define
		$$
		\hat z_{2j}(y) = \eta_{2, j}(y)\frac{1}{\log(2\rho^{-1} d_j)} Z_{0j}(y) g_j(y),
		$$
		where, for simplicity, we have denoted 
		\begin{equation}\label{defgj}
			g_{j}(y) = G_{\lambda \rho}(F_{j,\rho}(y), F_{j,\rho}(\xi_j')).
		\end{equation}

		Thus, we define 
		\begin{align}\label{aprox}
			\tilde z = \sum_j b_j \tilde z_j, \quad \mbox{with} \quad \tilde z_j= \eta_{1j} Z_{0j} + (1 - \eta_{1j}) \hat z_{2j}, 
		\end{align}
		with $b_j \in \R$ given by the expression
		\begin{equation}
			b_j \int_{\Omega_\rho} \chi_j |Z_{0j}|^2 + \int_{\Omega_\rho} \chi_j Z_{0j} \phi = 0.
		\end{equation}

		\medskip

		\noindent
		\textsl{2.- Application of the Lemma~\ref{lema1}:} 
		By linearity
		$$
		L \tilde \phi = h + L \tilde z \ \mbox{in} \ \Omega_\rho; \quad \RR_{\lambda \rho} \tilde \phi = \RR_{\lambda \rho} \tilde z \ \mbox{on} \ \partial \Omega_\rho,
		$$
		and by the choice of $b_j$, we are in position to use Lemma~\ref{lema1}.
		Then, we see that
		$$
		\| \tilde \phi \|_\infty \leq C \Big{(} \| h \|_* + \sum_j |b_j| ( \| L \tilde z_j \|_* + \frac{1}{\lambda \rho} \| \RR_{\lambda \rho} \tilde z_j \|_{L^\infty(\partial \Omega_\rho)}) \Big{)}, 
		$$
		from which we get
		\begin{equation}\label{sol}
			\|  \phi \|_\infty \leq C \Big{(} \| h \|_* + \sum_j |b_j| (\| \tilde z_j\|_\infty + \| L \tilde z_j \|_* + \frac{1}{\lambda \rho} \| \RR_{\lambda \rho} \tilde z_j \|_{L^\infty(\partial \Omega_\rho)}) \Big{)}.
		\end{equation}
		
		Now we estimate each term in the right-hand side above, stating with the terms concerning $\tilde z_j$. 
		
		\medskip
		
		\noindent
		\textsl{3.- Estimates for $g_j$ in~\eqref{defgj}:} We start with some estimates concerning the function $g_j$, consequence of~\eqref{G} and the properties of the conformal map $F_j$.
		
		First, by~\eqref{cotaF'}, for all $y \in \Omega_\rho$ with $|y -\xi_j'| \geq R$, we have
		\begin{equation*}
			\left|\int_{0}^{+\infty} e^{-\lambda \rho s}\frac{(F_{j, \rho}(y))_2 + s + (F_{j, \rho}(\xi_j'))_2}{|F_{j, \rho}(y) + s e_2- F_{j, \rho}(\xi_j')^*|^2} ds \right|\leq \int_{0}^{+\infty} \frac{\rho e^{-\lambda \rho s} ds}{|s + d_j + O(\lambda^{-2})|} \leq C,
		\end{equation*}
		for some $C > 0$. Using~\eqref{conforme},~\eqref{cotaF'}, we have
		\begin{align*}
			& \Gamma(F_{j, \rho}(y) - F_{j,\rho}(\xi_j')) = O(\log|y - \xi_j'|),
		\end{align*}
		and
		\begin{equation*}
			\Gamma(F_{j, \rho}(y) - F_{j,\rho}(\xi_j')^*) = \left \{ \begin{array}{ll} \log(2\rho^{-1}d_j) + O(R \lambda \rho) \quad & \mbox{if} \ |y - \xi_j'| \leq R + 1, \\ O(\log |y - \xi_j'|) + \log(2 d_j \rho^{-1}) \quad & \mbox{if} \ |y - \xi_j'| > R + 1, \end{array} \right .
		\end{equation*}
		where the $O$ terms are independent of $\epsilon, \lambda$. Moreover, in the particular case $y \in \partial \Omega_\rho$, since $F_{j,\rho}(y)$ belongs to $\partial \R_+^2$, there exists $C > 0$ such that
		\begin{equation}\label{ginfty}
			\left| \frac{g_j}{\log(2d_j\rho^{-1})}\right|_{L^\infty(\Omega_\rho \setminus B_{R}(\xi_j'))} \leq C.
		\end{equation}
		
		Additionally, a direct computation leads us to
		\begin{align*}
			\nabla g_j(y) = & O(|y - \xi_j'|^{-1}) + O(|F_{j,\rho}(y) - F_{j, \rho}(\xi_j')^*|^{-1}) \\
			& + O(1) \int_{0}^{+\infty} e^{-\lambda \rho s} \frac{ds}{|se_2 - F_{j, \rho}(\xi_j')^*|^2} \\
			= & O(|y - \xi_j'|^{-1}) + O(\rho \lambda).
		\end{align*}
		
		In fact, we have that for $|y - \xi_j'| = R$, using that $F_j(\hat \xi_j)$ is a rotation, the following expansion takes place
		\begin{equation}\label{nablagj}
			\nabla g_j(y) = \frac{y - \xi_j'}{|y - \xi_j'|^2} \Big{(} 1 + O(\lambda^{-1}) \Big{)} + O(\rho \lambda). 
		\end{equation}
		
		\medskip
		
		Now we look for the terms involving $\tilde z_j$ in~\eqref{sol}. By the previous discussion, it is easy to see the existence of $C > 0$ not depending on $\epsilon, \lambda$ such that
		\begin{equation}\label{zinfty}
			\| \tilde z_j \|_\infty \leq C.
		\end{equation}
		
		\medskip
		
		Now we deal with $\| L \tilde z_j\|_*$. It is clear that $L\tilde z_j = O(\rho^3)$ in $B_R(\xi_j')$. 
		
		Thus, we concentrate on the case $|y - \xi_j'| \geq R$. With the above estimates, for $R \leq |y - \xi_j'| \leq R+1$ (recalling that since $F$ is conformal we have $g_j$ is harmonic there) we have
		\begin{align*}
			L \tilde z_j(y) = & L \left(\eta_1 Z_{0j} \left(1 - \frac{g_j}{\log(2 d_j \rho^{-1})}\right) \right)(y) + L \hat z_j(y) \\
			= & O(\rho^3) + O\left(\frac{1}{|\log (\lambda \rho)|}\right) 
		\end{align*}
		
		For $|y - \xi_j'| \geq R + 1$, we see that
		\begin{align*}
			L \tilde z_j (y) = O \Big{(} \frac{\log|y - \xi_j'|}{|y - \xi_j'|^3 |\log(\lambda \rho)|} \Big{)}.
		\end{align*}

		From which we conclude 
		\begin{equation}\label{Lz*}
			\| L\tilde z_j\|_* \leq \frac{C R^{2 + \sigma}}{|\log(\lambda \rho)|},
		\end{equation}
		for some $C > 0$. 
		
		\medskip
		
		Now we concentrate on $\RR_{\lambda \rho} \tilde z_j$. Using the properties of $\eta_{2j}$, we have
		\begin{equation}\label{Rz}
			\RR_{\lambda \rho} \tilde z_j = \frac{1}{|\log(\lambda \rho)|} \eta_{2j} \Big{(}  Z_{0j} \RR_{\lambda \rho} g_j +  \frac{\partial Z_{0j}}{\partial \nu} g_j \Big{)} \quad \mbox{in} \ \partial \Omega_\rho.
		\end{equation}
		
		Notice that $\RR_{\lambda \rho} \tilde z_j(y) = 0$ for $|y - \xi_j'| \geq r/\rho$, so we concentrate on the analysis of $y \in \partial \Omega_\rho$ with $|y - \xi_j'| < r/\rho$. For simplicity, we denote this region as $A_j$. We have $g_j$ is uniformly bounded on $\partial \Omega_\rho$.  In fact, notice that if $y \in \partial \Omega_\rho$, we have
		$|F_{j, \rho}(y) - F_{j, \rho}(\xi_j')| = |F_{j, \rho}(y) - F_{j, \rho}(\xi_j')^*|$ and thus, using that $F_{j \rho}(\xi_j')_2 \geq c/(\lambda \rho)$ we have
		$$
		|g_j(y)| \leq C \int_{0}^{+\infty} \frac{e^{-\lambda \rho s}}{|s e_2 - (F_{j\rho}(\xi_j'))^*|} ds \leq C.
		$$
		
		Using this and the explicit formula for $Z_{0j}$, we conclude that
		\begin{equation*}
			\RR_{\lambda \rho} \tilde z_j = \frac{\eta_{2j}}{|\log(\lambda \rho)|} (O(1) \RR_{\lambda \rho} g_j + O(\lambda \rho)).
		\end{equation*}
		
		Noticing that
		\begin{align*}
			\RR_{\lambda \rho} g_j(y) 
			& = c \left(-\frac{\partial G}{\partial x_2}(F_{j\rho}(y), F_{j\rho}(\xi_j')\right)  + \lambda \rho G(F_{j\rho}(y), F_{j\rho}(\xi_j')) \\
			& = \lambda \rho (c - 1) g_j(y),
		\end{align*}
		where $c = 1 + O(\rho y)$ is the conformal factor of the map $F$. Thus, for each $y \in \partial \Omega_\rho$ with $|y - \xi_j'| \leq r/\rho$, we get
		\begin{equation}\label{cotaRg}
			\RR_{\lambda \rho} g_j(y) = O(\rho y) \lambda \rho g_j(y),
		\end{equation}
		and from here, replacing in~\eqref{Rz}, we conclude that
		\begin{equation}\label{Rginfty}
			\| \RR_{\lambda \rho} \tilde z_j \|_\infty \leq C \frac{\lambda \rho}{|\log(\lambda \rho)|}.
		\end{equation}
		
		In particular, using this estimate together with~\eqref{Lz*} and replacing them into~\eqref{lluvia} we have
		\begin{equation}\label{lluvia2}
			\|  \tilde \phi \|_\infty \leq C (\| h \|_* + \frac{1}{|\log(\lambda \rho)|}\sum_j |b_j|),
		\end{equation}
		meanwhile, using~\eqref{zinfty},~\eqref{Lz*} and~\eqref{Rginfty} and replacing them into~\eqref{sol}, we get
		\begin{equation}\label{sol2}
			\| \phi \|_\infty \leq C (\| h \|_* + \sum_{j} |b_j|),
		\end{equation}
		where in both inequalities the constant $C > 0$ depends on $R$, but not on $\lambda$ nor $\epsilon$.

		\medskip
		
		\noindent
		\textsl{4.- Estimate for $b_j$:} Using the equation for $\tilde \phi$, multiplying by $\tilde z_j$ and integrating by parts, we conclude that 
		\begin{equation}\label{testb}
			\begin{split}
				b_j \int_{\Omega_\rho} \tilde z_j L \tilde z_j & =  \int_{\Omega_\rho} \tilde \phi L \tilde z_j  - \int_{\partial \Omega_\rho} \tilde \phi \RR_{\lambda \rho} \tilde z_j + b_j \int_{\partial \Omega_\rho} \tilde z_j \RR_{\lambda \rho} \tilde z_j - \int_{\Omega_\rho} h \tilde z_j \\
				& =:  I_1 + I_2 + b_j I_3 + I_4.
			\end{split}
		\end{equation}
		for each $j$. Now we proceed to estimate each term.
		
		It is direct to see that
		\begin{equation*}
			|I_4| \leq C \| h \|_*, 
		\end{equation*}
		and using~\eqref{Lz*} and~\eqref{lluvia2}, we get
		\begin{equation*}
			|I_1| \leq C \| \tilde \phi \|_\infty R^{-\sigma} \| L \tilde z_j \|_* \leq 
			C \frac{R^2}{|\log(\lambda \rho)|} (\| h \|_* + \frac{R^{2 + \sigma}}{|\log(\lambda \rho)|}\sum_j |b_j|).
		\end{equation*}

		Now, for $I_2$ we make
		\begin{equation*}
			I_2 \leq \frac{C}{|\log(\lambda \rho)|} \| \tilde \phi\|_\infty \int_{\partial \Omega_\rho} |\RR_{\lambda \rho} g_j| =: \frac{C}{|\log(\lambda \rho)|}\| \tilde \phi\|_\infty \tilde I_2.
		\end{equation*}
		
		By~\eqref{cotaRg} we have
		\begin{equation*}
			\tilde I_2 
			\leq  C \lambda \rho \int_{A_j} |g_j(y)| ds(y),
		\end{equation*}
		and using the explicit formula~\eqref{G}, we get that
		\begin{align*}
			\tilde I_2 \leq & C \lambda \rho\int_{A_j} \int_{0}^{+\infty} \frac{e^{-\lambda \rho s} |s + F_{j\rho}(\xi_j')_2|}{|F_{\lambda \rho}(y) + se_2 - F_{j\rho}(\xi_j')^*|^2} ds ds(y).
		\end{align*}
		
		From here, using the estimate $F_{j\rho}(\xi_j') = \frac{1}{\lambda \rho}(O(\frac{1}{\lambda}), 1 + O(\frac{1}{\lambda}))$, by a change of variables, we can find universal constants $C, c > 0$ such that
		\begin{align*}
			\tilde I_2 \leq & C \lambda \rho \int_{0}^{1/\rho} \int_{0}^{+\infty} \frac{e^{-c \lambda \rho s} |s + \frac{1}{\lambda \rho}|}{t^2 + (s + \frac{1}{\lambda \rho})^2} ds dt, 
		\end{align*}
		by taking $\lambda$ suitably large. Thus, by algebraic manipulations, we conclude that
		$
		\tilde I_2 \leq C
		$
		for some $C > 0$ not depending on $R, \epsilon, \lambda$.
		From here conclude that
		\begin{equation*}
			|I_2| \leq \frac{C}{|\log(\lambda \rho)|} (\| h \|_* + \frac{R^{2 + \sigma}}{|\log(\lambda \rho)|}\sum_j |b_j|),
		\end{equation*}
		and by similar estimates, we conclude
		\begin{equation*}
			|I_3| \leq \int_{\partial \Omega_\rho} | \tilde z_j \RR_{\lambda \rho} \tilde z_j| ds(y) = O \Big{(} \frac{1}{|\log(\lambda \rho)|^2} \Big{)}.
		\end{equation*}
		
		Collecting the previous estimates and replacing them int~\eqref{testb}, we conclude that
		\begin{equation}\label{bjI0}
			b_j I_0 = O(1) \Big{(} \| h \|_* + \frac{1}{|\log(\lambda \rho)|^2 } \sum_j |b_j| \Big{)},
		\end{equation}
		where $I_0 := \int_{\Omega_\rho} \tilde z_j L \tilde z_j$ and $O(1)$ depends on $R$, but not on $\epsilon, \lambda$.
		
		We claim that
		\begin{equation}\label{I0}
			|I_0| \geq \frac{c}{|\log(\lambda \rho)|} 
		\end{equation}
		for some $c > 0$ independent of $R, \lambda$ and $\epsilon$. If we assume this is 
		true, we replace this into~\eqref{bjI0} and by taking $\lambda \rho$ small enough in terms of $R$, we arrive at
		$$
		|b_j| \leq C |\log(\lambda \rho)| \| h \|_*,
		$$
		and replacing this into~\eqref{sol2} we have been proved the lemma.
		Thus, it remains to get the
		
		\medskip
		\noindent
		{5.- Proof of estimate~\eqref{I0}:} We divide the integral in the regions $|y - \xi_j'| \leq R, R < |y - \xi_j'| \leq R + 1$, and $R+1 < |y - \xi_j'|$. Then, using the estimates~\eqref{ginfty}, the gradient of $g_j$,~\eqref{zinfty}, and the fact that for each $i \neq j$ we have $|y - \xi_i'| \geq r/(2\rho)$ in the support of $\eta_{2j}$, we have  
		\begin{equation}\label{I00}
			I_0 = O(R^2\lambda^4 \rho^4) + I_{01} + O\left(R^{-3} \frac{1}{|\log(\lambda \rho)|}\right),
		\end{equation}
		where 
		$$
		I_{01} = \int_{B_{R+1}(\xi_j') \setminus B_R(\xi_j')} L\tilde z_{0j} \tilde z_{0j}. 
		$$
		
		Writing $\tilde z_{0j} = \eta_{1,j}(Z_{0j} - \hat z_{0j}) + \hat z_{0j}$ valid in the annular region $A_R = \{ y: R < |y - \xi_j'| < R + 1 \}$, we can write
		\begin{align*}
			I_{01} =&  \int_{A_R} \Big{(} \Delta \eta_{1j} (Z_{0j} - \hat z_{0j}) + 2 \nabla \eta_{1j} \nabla (Z_{0j} - \hat z_{0j}) \Big{)} \tilde z_{0j} \\
			& + \int_{A_R} \eta_{1j} L Z_{0j} \tilde z_{0j} + \int_{A_R} (1 - \eta_{1j}) L \hat z_{0j} \tilde z_{0j}.
		\end{align*}
		
		By similar arguments as in~\eqref{I00}, we have
		\begin{align*}
			I_{01} = &  \int_{A_R} \Big{(} \Delta \eta_{1j} (Z_{0j} - \hat z_{0j}) + 2 \nabla \eta_{1j} \nabla (Z_{0j} - \hat z_{0j}) \Big{)} \tilde z_{0j} + O\left(R^{-3} \frac{1}{|\log(\lambda \rho)|}\right) \\
			=: & I_{02} + O\left(R^{-3} \frac{1}{|\log(\lambda \rho)|}\right).
		\end{align*}
		
		Then, for $I_{02}$, we integrate by parts and using that $(1 - g_j/\log(2d_j \rho^{-1})) = O(\frac{\log(R)}{|\log(\lambda \rho)|})$ in $A_R$, we can get
		\begin{align*}
			I_{02}  = & -\int_{|y - \xi_j'| = R} \tilde z_{0j} \nabla (Z_{0j} - \hat z_{0j}) \nu \\
			&  - \int_{A_R} \eta_{1j}(\Delta(Z_{0j} - \hat z_{0j}) \tilde z_{0j} + \nabla (Z_{0j} - \hat z_{0j}) \nabla \tilde z_{0j}) \\
			& - \int_{A_R} \nabla \eta_{1j} \nabla \tilde z_{0j} (Z_{0j} - \hat z_{0j}) \\
			=: & -I_{03} + O\left(\frac{R^{-2}}{|\log(\lambda \rho)|}\right) + O\left(\frac{\log^2(R)}{|\log(\lambda \rho)|^2}\right)+ O\left(\frac{R^2 \log^2(R)}{|\log(\lambda \rho)|^2}\right).
		\end{align*}
		
		For $I_{03}$ we have
		\begin{align*}
			-I_{03} = & -\int_{|y - \xi_j'| = R} \tilde z_{0j} (1 - g_j/\log(2d_j\rho^{-1})) \nabla Z_{0j} \nu \\ & + \frac{1}{\log(2d_j\rho^{-1})}\int_{|y - \xi_j'| = R} Z_{0j}^2 \nabla g_j \nu \\
			= & O\left(\frac{R^{-2} \log(R)}{\log(\lambda \rho)}\right) + \frac{Z_{0j}^2(R)}{\log(2 d_j \rho^{-1})}  \int_{|y - \xi_j'| = R} \nabla g_j \nu,
		\end{align*}
		from which, using the expansion~\eqref{nablagj} we arrive at
		\begin{equation*}
			I_{03} = -\frac{2 \pi Z_{0j}^2(R)}{\log(2 d_j \rho^{-1})} (1 + O(\lambda^{-1})) + O\left(\frac{R^{-2} \log(R)}{\log(\lambda \rho)}\right).
		\end{equation*}

		From here, taking $R, \lambda$ large enough we conclude
		$$
		|I_0| \geq \frac{\pi}{\log(2d_j \rho^{-1})} + O\left(\frac{R^{-2} \log(R)}{\log(\lambda \rho)}\right),
		$$
		from which~\eqref{I0} follows by fixing $R$ large enough independent of $\lambda, \rho$.
	\end{proof}
	
	Now we are in position to provide the
	\noindent
	\begin{proof}[Proof of Proposition~\ref{inliop}]
		Consider the Hilbert space
		$$
		H = \{ u \in H^1(\Omega_\rho) : \displaystyle\int_{\Omega_\rho } \chi_j Z_{ij}\phi=0\ \ \mbox{for}\ i=1,2,\ j=1,\dots,m \}
		$$
		endowed with the inner product
		$$
		\langle u, v \rangle_H = \int_{\Omega_\rho} \nabla u \cdot \nabla v + \lambda \rho \int_{\partial \Omega_\rho} uv.
		$$
		
		For each $\ell \in H^*$, we consider the map $R: H^* \to H$, given by the Riesz Representation Theorem, as 
		$$
		\langle R \ell, \varphi \rangle = \ell (\varphi), \quad \mbox{for all} \ \varphi \in H,
		$$
		and with this, we denote the operators
		$
		K_1, K_2 : H \to H,
		$
		given by
		$$
		K_1 (f) = R (f), \quad K_2(f) = R(Wf), \quad f \in H,
		$$
		where we have performed the identification $f \in L^2(\Omega_\rho)$ with the map $\varphi \mapsto \int_{\Omega_\rho} f \varphi \in H^*$. In fact, because of this identification we can assume that both operators are compact.
		
		For $h \in L^\infty(\Omega_\rho) \cap H^1(\Omega_\rho)$, let $c_{ij}$ the unique constants such that
		$$
		\tilde h := h + \sum_{ij} c_{ij} \chi_j Z_{ij} \in H.
		$$ 
		
		A simple computation tells us that for each $i, j$ we have
		$$
		\frac{-\int_{\Omega_\rho} h \chi_j Z_{ij}}{\int_{\Omega_\rho} \chi_j Z_{ij}^2} = c_{ij},
		$$
		from which, there exists $C_R > 0$ such that
		$$
		|c_{ij}| \leq C_R \| h \|_*
		$$
		
		The solvability of the problem is equivalent to the existence of a function $\phi \in H$ such that
		$$
		\phi - K_2(\phi) = K_1 \tilde h,
		$$1
		and by Fredholm alternative, we have a unique solution to this equation if the homogeneous problem
		$$
		\phi - K_2(\phi) = 0,
		$$
		has no nontrivial solutions. This is the case in view  of the estimates of Lemma~\ref{lema2} this is the case, from which the well-posedness of~\eqref{eq:3.5}. By the estimates for $c_{ij}$, we conclude that the solution $\phi$ to this problem satisfies the estimate~\eqref{estphicij}.
		
		For~\eqref{estdphi}, by standard arguments we have $\phi_{kl} := \partial_{(\xi_l)_k} \phi$ satisfies 
		$$
		L \phi_{kl} = - \frac{\partial W}{\partial (\xi_l)_k} \phi + \frac{\partial (c_{kl} \chi_l Z_{kl})}{\partial (\xi_l)_k}  \quad \mbox{in} \ \Omega_\rho,
		$$
		together with homogeneous Robin boundary conditions. In order to use the estimates in Lemma~\ref{lema2} we consider a function with the form
		$$
		\phi_{kl} + \tilde c_{kl} \chi_l Z_{kl},
		$$
		with $\tilde c_{kl}$ such that the orthogonality condition with respect to $Z_{kl}$ is satisfied. We omit the details.
	\end{proof}

	\section{Nonlinear Problem}\label{nonlinear}

	\begin{prop}\label{nlp}
		There exist $\lambda_0>0$ and $\epsilon_0>0$ such that for all $\lambda\geq\la_0$ and $\epsilon>0$ with $0<\rho\lambda^9\log\lambda<\e_0$ and for any points $\xi=(\xi_1,\dots, \xi_m)$ satisfying \eqref{setm}, the problem of finding a function $\phi$ and coefficients $c_{ij}(\xi)\in\R$, $i=1,2$, $j=1,\dots,m$ satisfying
		\begin{equation}\label{nline1}
			\begin{cases}
				\ds L(\phi)=-( R + \Lambda(\phi) + N(\phi))+\sum_{i=1}^2\sum_{j=1}^m c_{ij}\chi_j Z_{ij} &\mbox{in }\Omega_\rho\\ 
				\ds \frac{\partial\phi}{\partial\nu}+\rho\lambda\phi=0, &\mbox{on }\partial\Omega_\rho\\[0.4cm]
				\ds \int_{\Omega_\rho } \chi_jZ_{ij} \phi=0\ \ \ \  {\rm
					for}\ i=1,2,\ j=1,\dots,m,
			\end{cases}
		\end{equation}
		has a unique solution $\phi$ and scalars $c_{ij}$,
		$i=1,2$,\ $j=1,2,\dots,m$ satisfying
		$$
		\|\phi\|_{L^{\infty}(\Omega_\rho )}\leq C\rho\la^9\log\lambda |\log(\rho\la)|,\qquad |c_{ij}| \leq C\rho\la^9\log\lambda,
		$$
		where $L$, $R$, $\Lambda(\phi)$ and $N(\phi)$ are given by \eqref{L}, \eqref{R}, \eqref{lam} and \eqref{N} respectively. Moreover, the map
		$\xi'\mapsto\phi$ into the space $C(\bar{\Omega}_\rho)$, the
		derivative $D_{\xi'}\phi$ exists and defines a continuous function
		of $\xi'$, and there is a constant $C>0$, such that
		\begin{eqnarray}\label{cotadphi}
			\|D_{\xi'}\phi\|_{L^{\infty}(\Omega_\rho )}\leq
			C\rho\la^9\log\lambda|\log(\rho\la)|^2.
		\end{eqnarray}
	\end{prop}
	
	\begin{proof}
		Observe that in terms of the operator $T$ problem \eqref{nline1} becomes
		\begin{equation}\label{aphi}
			\phi=T\lf(-[R+\Lambda(\phi)+N(\phi)]\rg):=\mathcal{A}(\phi),
		\end{equation}
		where $T$ is the continuous linear map defined on the set of all $h\in L^\infty(\Om_\rho)$ satisfying $\|h\|_*<+\infty$, so that $\phi=T(h)$ correspond to the unique solution of the problem \eqref{eq:3.5}. For a given number $\nu>0$, let us consider
		$$
		\ml{F}_\nu = \{\phi\in C(\bar\Om_\rho) : \| \phi \|_\infty \le
		\nu \rho\la^9\log\lambda|\log(\rho\la)|\}
		$$
		From the Proposition \ref{inliop}, we get
		\begin{equation*}
			\begin{split}
				\|\ml{A}(\phi)\|_\infty & \le C | \log (\rho\la) |\lf[ \|R \|_*+ \| \Lambda(\phi) \|_* + \|N(\phi)\|_*\rg].
			\end{split}
		\end{equation*}
		From Lemma \ref{rest} it follows the estimate $\| R \|_*\le C\rho\la^9\log\lambda$. Furthermore,
		\begin{equation*}
			\| \Lambda(\phi)\|_*\le C\rho\la^9\log\lambda \|\phi\|_\infty\quad\text{and}\quad \| N (\phi) \|_* \le C \|\phi\|_\infty^2
		\end{equation*}
		Hence, we get for any $\phi\in \ml{F}_\nu$,
		\begin{equation*}
			\begin{split}
				\|\ml{A}(\phi)\|_\infty & \le C \rho\la^9\log\lambda |\log(\rho\la)| \lf[1 + \nu\rho\la^9 \log\lambda |\log(\rho\la)| + \nu^2\rho\la^9\log\lambda |\log(\rho\la)|^2\rg].
			\end{split}
		\end{equation*}
		Given any $\phi_1,\phi_2\in\ml{F}_\nu$, we have that
		\begin{equation*}
			\|\Lambda (\phi_1) - \Lambda(\phi_2)\|_* \le C \rho\la^9\log\lambda \,\|\phi_1-\phi_2\|_\infty
		\end{equation*}
		and
		\begin{equation*}
			\begin{split}
				\|N(\phi_1)-N(\phi_2)\|_* & \le C (\|\phi_1\|_\infty +
				\|\phi_2\|_\infty) \|\phi_1-\phi_2\|_\infty\\
				& \le C \nu \rho\la^9\log\lambda |\log(\rho\la)|\,\|\phi_1-\phi_2\|_\infty
			\end{split}
		\end{equation*}
		with $C$ independent of $\nu$. Therefore, from the Proposition
		\ref{inliop}
		\begin{equation*}
			\begin{split}
				\|\ml{A}(\phi_1)-\ml{A}(\phi_2)\|_\infty & \le C \nu \rho\la^9\log\lambda |\log (\rho\la)|^2 \|\phi_1-\phi_2\|_\infty,
			\end{split}
		\end{equation*}
		so that it follows that for all $\e$ sufficiently small and $\la$ sufficiently large $\ml{A}$ is a
		contraction mapping of $\ml{F}_\nu$ (for $\nu$ large enough), and
		therefore a unique fixed point of $\ml{A}$ exists in $\ml{F}_\nu$.
		\medskip
		
		Let us now discuss the differentiability of $\phi$ depending on
		$\xi'$, i.e., $\xi'\mapsto \phi(\xi')\in C(\bar\Om_\rho)$ is $C^1$.
		Since $R$ depends continuously (in the $*$-norm) on $\xi'$, using
		the fixed point characterization \eqref{aphi}, we deduce that the
		mapping $\xi'\mapsto \phi$ is also continuous. Then, formally
		\begin{equation*}
			\begin{split}
				\fr_{\xi'_{kl}} [\Lambda(\phi) ] = &\,\Big\{\rho^2\e^2\big[\partial_{\xi'_{kl}} (e^V)-\partial_{\xi'_{kl}} (e^{-V}) \big] - \partial_{\xi'_{kl}} W \Big\} \phi\\
				& \,
				+ [\rho^2\e^2(e^V+e^{-V}) -W] \partial_{\xi'_{kl}}\,\phi .
			\end{split}
		\end{equation*}
		and
		\begin{equation*}
			\begin{split}
				\fr_{\xi'_{kl}} [N(\phi) ] = &\,\rho^2\e^2\partial_{\xi'_{kl}} (e^V) (e^\phi -\phi
				-1)
				+ \rho^2\e^2e^V[ e^\phi -1] \partial_{\xi'_{kl}}\,\phi\\
				& \, -\rho^2\e^2\partial_{\xi'_{kl}} (e^{-V}) (e^{-\phi} +\phi -1)
				+ \rho^2\e^2e^{-V} [ e^{-\phi} -1] \partial_{\xi'_{kl}}\,\phi.
			\end{split}
		\end{equation*}
		From the definition of $V$ it follows that
		$$\fr_{\xi'_{kl}}V(y)=a_k\frac{4(y-\xi_k')_l}{\mu_k^2 + |y - \xi_k'|^2} - \sum_{i=1}^m a_i \frac{2\fr_{\xi'_{kl}}(\mu_i^2)}{\mu_i^2 + |y-\xi_i'|^2} + \rho \fr_{2l}H_\la(\rho y, \rho \xi_k') + O(\rho^2\la^2).$$
		Hence, $\|\rho^2\e^2\fr_{\xi'_{kl}} (e^V)\|_*$ and $ \|\rho^2\e^2\fr_{\xi'_{kl}}(e^{-V})\|_*$ are uniformly bounded. Then, we conclude that
		\begin{equation*}
			\begin{split}
				\| \fr_{\xi'_{kl}} [\Lambda(\phi) ] \|_* \le  &\,\Big\|\rho^2\e^2\big[\partial_{\xi'_{kl}} (e^V)-\partial_{\xi'_{kl}} (e^{-V}) \big] - \partial_{\xi'_{kl}} W \Big\|_* \|\phi\|_\infty\\
				& \,
				+ [\| \rho^2\e^2(e^V+e^{-V}) -W \|_* \|\partial_{\xi'_{kl}}\,\phi \|_\infty.
			\end{split}
		\end{equation*}
		and
		\begin{equation*}
			\begin{split}
				\|\fr_{\xi'_{kl}}[ N(\phi) ] \|_* & \le C\|\phi\|_\infty^2 + C
				\|\phi\|_\infty\, \|\partial_{\xi'_{kl}}\phi \|_\infty\\
				& \le C \nu\rho\la^9\log\lambda |\log (\rho\la)|\lf[\nu\rho\la^9 \log\lambda|\log (\rho\la) | + \| \fr_{\xi'_{kl}}
				\phi\|_\infty\rg].
			\end{split}
		\end{equation*}
		Also, observe that we have
		$$\fr_{\xi'_{kl}}\phi= - (\fr_{\xi'_{kl}}T ) \lf( R + \Lambda(\phi) + N (\phi)\rg) - T\lf(\fr_{\xi'_{kl}}\lf[ R +\Lambda(\phi)  + N (\phi)\rg]\rg).$$
		So, using \eqref{estphicij} and \eqref{estdphi}, we get
		\begin{equation*}
			\begin{split}
				\|\fr_{\xi'_{kl}}\phi\|_\infty 
				& \le C |\log (\rho\la)| \,\Big[|\log (\rho\la)| \lf( \|R\|_* +\|\Lambda (\phi)\|_* +
				\|N (\phi)\|_*\rg) \\
				&\qquad  + \|\fr_{\xi'_{kl}}R \|_* + \| \fr_{\xi'_{kl}} [\Lambda(\phi) ] \|_* + \|\fr_{\xi'_{kl}}
				[N(\phi) ] \|_*\Big].
			\end{split}
		\end{equation*}
		Let us estimate $\|\fr_{\xi'_{kl}} R\|_*$. We know that
		$$\fr_{\xi'_{kl}} R(y)=\lap \fr_{\xi'_{kl}}V(y)+\rho^2\e^2(e^V+e^{-V})\fr_{\xi'_{kl}}V(y).$$
		From similar computations to
		deduce Lemma \ref{rest} it follows that  for
		any $l=1,2$:
		\begin{itemize}
			\item if $|y-\xi_j'|>\frac{\de}{\rho }$ for
			all $j=1,\dots,m$ then $$\fr_{\xi_{kl}'}R(y)=O\big( \rho\la\log\lambda[\rho^4 + \rho^2\e^2]\big),$$
			\item if $\frac{\de}{\rho\la\log\la }\le |y-\xi_j'| \le \frac{\de}{\rho}$ for
			some $j\in\{1,\dots,m\}$ then 
			$$\fr_{\xi_{kl}'}R(y)=\rho\la^2\log^2\lambda\, O\bigg(\rho^2\epsilon^2\Big[\log^4\lambda+\frac{\la^4}{\rho^4|y-\xi_j'|^4}+ \rho^4\la^4|y-\xi_j'|^4\Big] \bigg)$$ and
			\item if $|y-\xi_j'|\le \frac{\de}{\rho\la\log\lambda}$ for some $j\in\{1,\dots,m\}$ then
			\begin{equation*}
				\begin{split}
					\fr_{\xi_{kl}'}R(y)
					=& \frac{8\mu_j^2}{(\mu_j^2+|y-\xi_j'|^2)^2} O(\rho\la\log\la [1 + |y-\xi_j'|] ) \\
					& + \frac{4 \de_{jk} (y-\xi_k')_l + 2\fr_{\xi'_{kl}}(\mu_j^2)}{\mu_j^2 + |y - \xi_j'|^2}O(\rho^2\e^2) +O(\rho^3\e^2\la\log\lambda + \rho^4).
				\end{split}
			\end{equation*}
		\end{itemize}
		Therefore, from the definition of *-norm we conclude that
		$$\|\fr_{\xi'_{kl}} R\|_*\le C(\rho\la \log\lambda + \rho^2\la^{11}\log^3\lambda).$$
		Similar computations as above and those used to deduce Lemma \ref{rest} lead us to find the estimate
		$$\big\|\rho^2\e^2(e^V - e^{-V})\fr_{\xi'_{kl}} V -\fr_{\xi'_{kl}} W \big\|_*\le C(\rho\la\log\lambda + \rho^2\la^{11}\log^3\lambda).$$
		Hence, we find the following
		estimate
		\begin{equation*}
			\begin{split}
				\|\fr_{\xi'_{kl}}\phi\|_\infty & \le C \big[ \rho\la^9\log\lambda |\log(\rho\la)|^2 + \rho\la^9\log\lambda |\log(\rho\la)|^2 \|\fr_{\xi'_{kl}}\phi\|_\infty\big]\\
			\end{split}
		\end{equation*}
		Thus, we conclude \eqref{cotadphi}. Note that $\fr_{\xi_{kl}'}\mu_j=O(\rho\lambda\log\lambda)$.
		
		The above computations can be made rigorous by using the implicit
		function theorem and the fixed point representation \eqref{aphi}
		which guarantees $C^1$ regularity in $\xi'$.
	\end{proof}
	
	\section{Variational reduction and energy computations}\label{energyex}
	
	In view of Proposition~\ref{nlp} we obtain a solution to~\eqref{line} with the form $V + \phi$ if we are able to get that
	\begin{equation}\label{scij}
		c_{ij}(\xi)=0,\qquad i=1,2,\ j=1,\dots,m.
	\end{equation}
	
	This problem is equivalent to look for critical points of the following functional
	$$
	F_{\e,\la}(\xi)=J_{\e,\la}(U+\tilde \phi),
	$$
	where $U$ is the approximation defined in \eqref{defU} and $\ds\tilde\phi(x)=\phi\Big(\frac{x}{\rho}\Big)$ with $\phi$ the solution to \eqref{line1}.
	The energy function $J_{\e,\la}$ is given by
	\begin{equation}\label{Jel}
		J_{\e,\la}(u) = \frac{1}{2} \int_\Omega | \nabla u |^2dx -\epsilon^2
		\int_\Omega \left(e^{u}+e^{-u}\right)dx+\frac{\la}{2}\int_{\partial\Omega} u^2d\sigma, \quad u \in H^1(\Omega).
	\end{equation}
	
	Notice that critical points for $J_{\epsilon, \la}$ are weak solutions for~\eqref{eq} compare with \cite{barpis}. We have the following sufficient condition to have~\eqref{scij}.
	\begin{lema}\label{cpfc0}
		There exists $\la_0>0$ and $\e_0>0$ such that for any $\la\ge\la_0$ and $\e>0$ so that $0<\rho\la^9<\e_0$, if $\xi\in\Om^m$ is a critical point of $F_{\e,\la}$ satisfying \eqref{setm}  then $u=U(\xi)+\tilde\phi(\xi)$ is a critical point of $J_{\e,\la}$, that is, if
		$D_\xi F_{\e,\la}(\xi)=0$ then $\xi$ satisfies system \eqref{scij},
		i.e., $u$ is a solution to \eqref{eq}.
	\end{lema}
	
	\begin{proof}
		Define the energy functional $I_{\e,\la}$ associated to problem
		\eqref{line}, namely,
		$$
		I_{\e,\la}(v)=\frac{1}{2}\int_{\Om_\rho}|\grad
		v|^2 - \rho^2\e^2\int_{\Om_\rho} (e^{v}+e^{-v}) \,dy + \frac{\rho\la}{ 2}\int_{\fr\Om_\rho} v^2\,d\sigma(y).
		$$
		Let us differentiate the function $F_{\e,\la}(\xi)$ with respect to
		$\xi$. Since 
		$$I_{\e,\la}(V(\xi') + \phi(\xi')) = J_{\e,\la}(U(\xi) +
		\ti\phi(\xi)),$$
		we can differentiate directly $I_{\e,\la}(V+\phi)$
		(under the integral sign), so that integrating by parts 
		\begin{equation*}
			\begin{split}
				\fr_{\xi_{kl}}F_{\e,\la}(\xi) = &\,
				\frac{1}{\rho}DI_{\e,\la}(V+\phi)\lf[\fr_{\xi_{kl}'}V
				+\fr_{\xi_{kl}'}\phi\rg]\\
				=&\, -\frac{1}{\rho}\sum_{i=1}^{2}
				\sum_{j=1}^m c_{ij} \int_{\Om_\e} \chi_j
				Z_{ij}\,\lf[\fr_{\xi_{kl}'}V + \fr_{\xi_{kl}'}\phi\rg].
			\end{split}
		\end{equation*}
		From the results of the previous
		section, this expression defines a continuous function of $\xi'$,
		and hence of $\xi$. Let us assume that $D_\xi F_{\e,\la}(\xi)=0$. Then,
		from the latter equality
		$$
		\sum_{i=1}^{2} \sum_{j=1}^m c_{ij} \int_{\Om_\e} \chi_j
		Z_{ij}\,\lf[\fr_{\xi_{kl}'}V + \fr_{\xi_{kl}'}\phi\rg]=0,\qquad k=1,2,\;l=1,\dots,m.
		$$
		Using  \eqref{cotadphi} and $\fr_{\xi_{kl}'}V = 4\frac{a_k}{\mu_k} Z_{kl} + O\big(\rho\la\log\lambda \big)$,
		where $O(\rho\la\log\lambda)$ is in the $L^\infty$ norm, it follows that
		$$
		\sum_{i=1}^{2} \sum_{j=1}^m c_{ij} \int_{\Om_\e} \chi_j
		Z_{ij}\,\lf[Z_{kl} + o(1)\rg]=0,\qquad k=1,2,\; l=1,\dots,m.
		$$
		with $o(1)$ small in the sense of the $L^\infty$ norm as $\rho\la^9\log\lambda\to0$.
		The above system is diagonal dominant and we thus get $c_{ij}=0$
		for $i=1,2$, $j=1,\dots,m$.
	\end{proof}
	
Next result states an expansion of $F_{\epsilon, \la}$ in terms of $\varphi_m$. 
In order to have a more clearly the relation between $\epsilon \in (0,1)$ and $\lambda > 1$ such that they satisfying $\epsilon \la^{16} \leq \epsilon_0$ for some $\epsilon_0 < 1$, and recalling $\rho = \epsilon \la^{-2}$, we see that
\begin{equation}\label{estlogrola}
0 \geq \log(\rho \lambda) \geq \log \epsilon - \log \epsilon^{-1/16} \geq 2\log \epsilon,
\end{equation}
from which we state the estimates in terms of the leading expression at the logarithm.
	\begin{prop}\label{fju}
		The following expansions holds
		$$
		F_{\e,\la}(\xi)=-16\pi m +8\pi m\log 8 -16\pi m\log(\rho\la^2)-4\pi \varphi_m(\xi)+\theta_{\e,\la}(\xi),
		$$
		in $C^1$-sense, where
		\begin{equation}\label{vphi}
			\varphi_m(\xi)=\sum_{j=1}^m\lf[H_\la(\xi_j,\xi_j) + \sum_{i=1,i\ne j}^m a_i a_j G_\la(\xi_i,\xi_j)\rg],
		\end{equation}
$$
		|\theta_{\e,\la}(\xi)|=O\big(\epsilon^2 \lambda^{14} |\log\epsilon |^3\big),
		$$
		and
		$$
		|\grad\theta_{\e,\la} (\xi)|= O \big(\epsilon \lambda^{16}\log^4 \epsilon \big),
		$$
		uniformly on points $\xi=(\xi_1,\dots,\xi_m)\in \Om^m$ satisfying
		the constraints \eqref{setm}, as $\e\to0$ and $\lambda\to+\infty$.
	\end{prop}
	
	\begin{proof}
		First, we shall expand the energy functional $J_{\e,\la}$ evaluated in the ansatz $U$, namely, we give an asymptotic estimate of $J_{\e,\la}(U)$.
		
		\begin{claim}
			The following expansion does hold
			$$J_{\e,\la}(U)=-16\pi m +8\pi m\log 8 -16\pi m\log(\rho\la^2)-4\pi \varphi_m(\xi)+O(\rho\la\log\lambda)$$
		\end{claim}
		
		\begin{proof}
			First, we will evaluate the quadratic and boundary parts of energy evaluated at
			$U$, that is, integrating by parts
			$$\frac{1}{2}\int_\Om|\grad U|^2\,dx + \frac{\la}{2} \int_{\fr \Om} U^2\, d\sigma =-\frac{1}{2}\int_\Om U\Delta
			U\,dx=-\frac{1}{2}\sum_{j=1}^m a_j \int_\Om U\Delta
			U_j\,dx,$$
			since on $\partial\Om$
			$$\pd{U}{\nu} +\la U=0.$$ 
			Using the equation \eqref{eqwj} of $U_j=w_j+H_j$ (recall $H_j$ is harmonic), we have
			\begin{equation*}
				\begin{split}
					\int_\Om U(-\Delta U_j)\,dx
					&=\int_\Om \e^2 e^{w_j(x)}U(x)\,dx\\
					&= a_j\int_\Om \e^2 e^{w_j}(w_j+H_j) + \sum_{i\ne j} a_i\int_\Om \e^2 e^{w_j}(w_i+H_i).
				\end{split}
			\end{equation*}
			Then, we expand as follows
			$$\int_\Om \e^2 e^{w_j}(w_j+H_j) = \int_{B(\xi_j,\frac{d_j}{2})} \e^2e^{w_j}(w_j+H_j)+\int_{\Om\sm B(\xi_j,\frac{d_j}{2})} \e^2e^{w_j}(w_j+H_j).$$
			Using \eqref{estimacion1} and $x-\xi_j=\mu_j\rho y$, we obtain that
			\begin{equation*}
				\begin{split} 
					&\int_{B(\xi_j,\frac{d_j}{2})} \e^2e^{w_j}(w_j+H_j)\\
					=&\,\int_{B(\xi_j,\frac{d_j}{2})}\frac{8\mu_j^2\rho^2}{(\mu_j^2\rho^2+|x-\xi_j|^2)^2} \bigg[\log\frac{1}{(\mu_j^2\rho^2+|x-\xi_j|^2)^2}+ H_\la(x,\xi_j)\\
					&\qquad\qquad\qquad \qquad +O\lf(\frac{\mu_j^2\rho^2}{d_j^2}\rg)\bigg]\\
					=&\,\int_{B(0,\frac{d_j}{2\mu_j\rho})}\frac{8}{(1+|y|^2)^2} \bigg[-2 \log(1+|y|^2) - 4\log(\mu_j\rho) + H_\la(\xi_j+\mu_j\rho y,\xi_j)\\
					&\qquad\qquad\qquad \qquad +O\lf(\frac{\mu_j^2\rho^2}{d_j^2}\rg)\bigg]\\
					=&\,-16\pi -32\pi\log(\mu_j\rho)+8\pi H_\la (\xi_j,\xi_j)+O(\rho\la\log\lambda).
				\end{split}
			\end{equation*}
			
			Since $\mu_j$ is uniformly bounded and away from zero, and since $d_j = O(\la^{-1})$, we see that
			$$\int_{B(0,\frac{d_j}{2\mu_j\rho})}\frac{8}{(1+|y|^2)^2}\log(1+|y|^2)\, dy = 8\pi+O\Big(\frac{\mu_j^2\rho^2}{d_j^2}\big|\log\frac{\mu_j\rho}{d_j}\big|\Big),$$
			and
			$$\int_{B(0,\frac{d_j}{2\mu_j\rho})}\frac{8}{(1+|y|^2)^2}\, dy=8\pi+O\Big(\frac{\mu_j^2\rho^2}{d_j^2}\Big)$$
			and by using \eqref{tehxi}, we get $H_\la(\xi_j+\mu_j\rho y,\xi_j)=H_\la(\xi_j,\xi_j)+O(\mu_j\rho\la\log\la |y|)$ in $B(0,\frac{d_j}{ 2\mu_j\rho})$ so that
			$$\int_{B(0,\frac{d_j}{2\mu_j\rho})}\frac{8}{(1+|y|^2)^2}H_\la(\xi_j+\mu_j\rho y,\xi_j)\, dy=8\pi H_\la(\xi_j,\xi_j)+O(\rho\la\log\la ).$$
			Also,
			\begin{equation*}
				\begin{split} 
					\int_{\Om\sm B(\xi_j,\frac{d_j}{2})} \e^2e^{w_j}(w_j+H_j)
					&=\int_{\Om\sm B(\xi_j,\frac{d_j}{2})} \e^2e^{w_j}\lf[G_\la(x,\xi_j) +O\lf(\frac{\mu_j^2\rho^2}{d_j^2}\rg)\rg]\\
					&=O(\rho^2\la^2\log\la)
				\end{split}
			\end{equation*}
			in view of
			$$\e^2e^{w_j}=\frac{8\mu_j^2\rho^2}{(\mu_j^2\rho^2+|x-\xi_j|^2)^2}=O\Big(\frac{\mu_j^2\rho^2}{d_j^2}\Big)$$
			and $G_\la(x,\xi_j)=O(\log\la )$ for $|x-\xi_j|>\dfrac{d_j}{2}$. Therefore, we obtain that
			\begin{equation}\label{iewjuj}
				\begin{split}
					\int_\Om \e^2 e^{w_j}(w_j+H_j)
					&=-16\pi -32\pi\log(\mu_j\rho)+8\pi H_\la (\xi_j,\xi_j)+O(\rho\la\log\lambda)
				\end{split}
			\end{equation}
			Now, for $i\ne j$ we have that
			$$\int_\Om \e^2 e^{w_j}(w_i+H_i) = \int_{B(\xi_j,\frac{d_j}{2})} \e^2e^{w_j}(w_i+H_i)+\int_{\Om\sm B(\xi_j,\frac{d_j}{2})} \e^2e^{w_j}(w_i+H_i).$$
			Then, by using \eqref{estimacion1}, we get that
			\begin{equation*}
				\begin{split}
					\int_{B(\xi_j,\frac{d_j}{2})} \e^2 e^{w_j}(w_i+H_i) 
					=&\, 8\pi G_\la(\xi_i,\xi_j) + O(\rho\la\log\lambda)
				\end{split}
			\end{equation*}
			\begin{equation*}
				\begin{split}
					\int_{\Om\sm B(\xi_j,\frac{d_j}{2})} &\e^2 e^{w_j}(w_i+H_i)\\
					=&\,\int_{B(\xi_i,\frac{d_i}{2})}\e^2 e^{w_j}\left[\log\frac{1}{(\mu_i^2\rho^2+|x-\xi_i|)^2}+H_\la (x,\xi_i)+O\Big(\frac{\mu_i^2\rho^2}{d_i^2}\Big)\right] dx \\
					&\, +\int_{\Om\sm[B(\xi_j,\frac{d_j}{2})\cup B(\xi_i,\frac{d_i}{2})]}\e^2 e^{w_j}\left[G_\la(x,\xi_i)+O\Big(\frac{\mu_i^2\rho^2}{d_i^2}\Big)\right] dx\\
					=&\, O\lf(\rho^2\bigg[\frac{|\log \rho|}{\la^2} +\frac{|\log(\rho\la)|}{\la^2}\bigg]\rg)+O\left(\frac{\rho^2}{\lambda^2}\log\la +\frac{\rho^4}{\lambda^4}\right)\\
					=&\, O\left(\frac{\rho^2}{\lambda^2}|\log(\rho\la)|\right)
				\end{split}
			\end{equation*}
			in view of
			\begin{equation*}
				\begin{split}
					\int_{B(\xi_i,\frac{d_i}{2})}&\frac{8\mu_j^2\rho^2}{(\mu_j^2\rho^2+|x-\xi_j|^2)^2 }\left[\log\frac{1}{(\mu_i^2\rho^2+|x-\xi_i|)^2}+H_\la (x,\xi_i)+O\Big(\frac{\mu_i^2\rho^2}{d_i^2}\Big)\right] dx\\
					=&\,O\left(\rho^2\int_{B(\xi_i,\frac{d_i}{2})}|\log (\mu_i^2\rho^2+|x-\xi_i|)|\,dx +\rho^2\int_{B(\xi_i,\frac{d_i}{2})} |H_\la(x,\xi_i)|\, dx+\rho^4\right)\\
					=&\,O\left(\rho^2\Big[d_i^2|\log\rho| + \rho^2\Big(\frac{d_i}{\rho}\Big)^2\log\Big(\frac{d_i}{\rho}\Big)\Big]\right) + O(\rho^2[d_i^2\log\la + d_i^3\la \log\lambda ])
				\end{split}
			\end{equation*}
			and
			\begin{equation*}
				\begin{split}
					\int_{\Om\sm[B(\xi_j,\frac{d_j}{2})\cup B(\xi_i,\frac{d_i}{2})]}\frac{8\mu_j^2\rho^2}{(\mu_j^2\rho^2+|x-\xi_j|^2)^2 }&\left[G_\la(x,\xi_i)+O\Big(\frac{\mu_i^2\rho^2}{d_i^2}\Big)\right] dx\\
					=O\left(\frac{\rho^2}{d_j^2}\log\la+\frac{\rho^4}{d_j^2d_i^2}\right)
				\end{split}
			\end{equation*}
			Thus, we obtain that
			\begin{equation}\label{iewjui}
				\int_\Om \e^2 e^{w_j}(w_i+H_i)= 8\pi G_\la(\xi_i,\xi_j) + O(\rho\la\log\lambda).
			\end{equation}
			Taking into account \eqref{condmu}, \eqref{iewjuj} and \eqref{iewjui}, we find that
			\begin{equation}
				\begin{split}
					&\frac{1}{2}\int_\Om|\grad U|^2\,dx + \frac{\la}{2} \int_\Om U^2\, d\sigma=\frac{1}{2}\sum_{j=1}^m\left[\int_\Om \e^2e^{w_j}U_j+a_j\sum_{i=1,i\ne j}^m a_i\int_\Om\e^2e^{w_j}U_i\right]\\
					&=\frac{1}{2}\sum_{j=1}^m\lf[-16\pi -32\pi\log(\mu_j\rho)+8\pi H_\la (\xi_j,\xi_j) + a_j\sum_{i=1,i\ne j}^m a_i8\pi G_\la(\xi_i,\xi_j)\rg] + O(\rho\la\log\lambda)
				\end{split}
			\end{equation}
			
			On the other hand, we have
			$$\int_\Om \e^2 (e^{U} +e^{-U})\, dx=\sum_{j=1}^m \int_{B(\xi_j,\frac{d_j}{ \log\lambda})} \e^2(e^{U} +e^{-U})\,dx + \int_{\Om\sm \cup_{j=1}^m B(\xi_j,\frac{d_j}{\log\lambda})}\e^2(e^{U} +e^{-U})\,dx.$$
			Observe that
			\begin{equation*}
				\begin{split}
					\int_{\Om\sm \cup_{j=1}^m B(\xi_j,\de)}\e^2(e^{U} +e^{-U})\,dx&=\int_{\Om\sm \cup_{j=1}^m B(\xi_j,\de)}\e^2\big(e^{\sum_{i=1}^m a_i G_\la(x,\xi_i)+O(\rho^2\la^2)} \\
					&\qquad\qquad\quad+e^{-\sum_{i=1}^m a_i G_\la(x,\xi_i)+O(\rho^2\la^2)}\big)\,dx\\
					&=O(\e^2),
				\end{split}
			\end{equation*}
			by using that $\sum_{i=1}^m a_i G_\la(x,\xi_i)=O(1)$ in $\Om\sm \cup_{j=1}^m B(\xi_j,\de)$. So, we get that
			$$\int_{\Om\sm \cup_{j=1}^m B(\xi_j,\frac{d_j}{\log\lambda})} \e^2 (e^{U} +e^{-U})\, dx=\sum_{j=1}^m \int_{B(\xi_j,\de)\sm B(\xi_j,\frac{d_j}{\log\lambda})} \e^2(e^{U} +e^{-U})\,dx + O(\e^2).$$
			Now, assume that $a_j=1$ so that by \eqref{Hxi} we obtain 
			\begin{equation*}
				\begin{split}
					\int_{B(\xi_j,\de)\sm B(\xi_j,\frac{d_j}{\log\lambda})} \e^2e^{U}&=\int_{B(\xi_j,\de)\sm B(\xi_j,\frac{d_j}{\log\lambda})} \e^2e^{G_\la(x,\xi_j) + \sum_{l\ne j} a_lG_\la(x,\xi_l)+O(\rho^2\la^2)}\, dx\\
					&=  \int_{B(\xi_j,\de)\sm B(\xi_j,\frac{d_j}{\log\lambda})}\e^2\frac{e^{H_\la(x,\xi_j)+O(1)}}{|x-\xi_j|^4 }dx\\ &=O\lf(\e^2\la^4\int_{B(\xi_j,\de)\sm B(\xi_j,\frac{d_j}{\log\lambda})} |x-\xi_j|^{-4}\, dx\rg)\\
					&=O(\e^2\la^6\log^2\lambda)
				\end{split}
			\end{equation*}
			and
			\begin{equation*}
				\begin{split}
					\int_{B(\xi_j,\de)\sm B(\xi_j,\frac{d_j}{\log\lambda})} \e^2e^{-U}&=\int_{B(\xi_j,\de)\sm B(\xi_j,\frac{d_j}{\log\lambda})} \e^2e^{-G_\la(x,\xi_j) - \sum_{l\ne j} a_lG_\la(x,\xi_l)+O(\rho^2\la^2)}\, dx\\
					&=  \int_{B(\xi_j,\de)\sm B(\xi_j,\frac{d_j}{\log\lambda})}\e^2 |x-\xi_j|^4 e^{-H_\la(x,\xi_j)+O(1)} dx\\ &=O\lf(\e^2\la^4\int_{B(\xi_j,\de)\sm B(\xi_j,\frac{d_j}{\log\lambda})} |x-\xi_j|^{4}\, dx\rg)\\
					&=O(\e^2\la^4).
				\end{split}
			\end{equation*}
			Similarly, for $a_j=-1$ we find that
			$$\int_{B(\xi_j,\de)\sm B(\xi_j,\frac{d_j}{\log\lambda})}  \e^2(e^{U}+e^{-U})=O(\e^2\la^4+\e^2\la^6\log^2\lambda).$$
			Now, assume that $a_j=1$, so that by Lemma \ref{1} and taking $x-\xi_j=\mu_j\rho y$ we get
			\begin{equation*}
				\begin{split}
					\int_{B(\xi_j,\frac{d_j}{\log\lambda})} \e^2e^{U}\,dx
					&=\int_{B(\xi_j,\frac{d_j}{\log\lambda})}\frac{8\mu_j^2\rho^2}{(\mu_j^2\rho^2+|x-\xi_j|^2)^2}\,\exp\bigg(H_j(x)+\sum_{l\ne j}U_l(x)\bigg)\,dx\\
					&= \int_{B(\xi_j,\frac{d_j}{\log\lambda})}\frac{ 8e^{H_\la (x,\xi_j)-\log(8\mu_j^2)+4\log\la+\sum_{l\ne j}a_l G_\la (x,\xi_l)+O(\rho^2\la^2)}}
					{\mu_j^2\rho^2\lf(1+\big(\frac{|x-\xi_j|}{\mu_j\rho}\big)^2\rg)^2}\,dx\\
					&=\int_{B(0,\frac{d_j}{\mu_j\rho\log\lambda})}\frac{8 e^{H(\xi_j+\mu_j\rho y,\xi_j)-\log(8\mu_j^2)+4\log\la+\sum_{l\ne j}a_l G(\xi_j+\mu_j\rho	y,\xi_l)+O(\rho^2\la^2)}} {\lf(1+|y|^2\rg)^2}\,dy\\
					&=\int_{B(0,\frac{d_j}{\mu_j\rho\log\lambda})}\frac{8}{\lf(1+|y|^2\rg)^2}\left[1+O(\rho\la\log\lambda|y|+\rho^2\la^2)\right]\,dy\\
					&=8\pi + O(\rho\la\log\lambda),
				\end{split}
			\end{equation*}
			and
			\begin{equation*}
				\begin{split}
					\int_{B(\xi_j,\frac{d_j}{\log\lambda})} \e^2e^{-U}\,dx
					&=\e^2\int_{B(\xi_j,\frac{d_j}{\log\lambda})} \lf[\frac{8\mu_j^2\rho^2}{(\mu_j^2\rho^2+|x-\xi_j|^2)^2\e^2}\rg]^{-1}\,\exp\bigg(-H_j(x)-\sum_{l\ne j}U_l(x)\bigg)\,dx\\
					&= \int_{B(\xi_j,\frac{d_j}{\log\lambda})}\frac{ \e^4}
					{8\mu_j^2\rho^2}\lf(\mu_j^2\rho^2+ |x-\xi_j|^2\rg)^2 \\
					&\qquad \qquad \qquad \times e^{-H_\la (x,\xi_j)+\log(8\mu_j^2)-4\log\la-\sum_{l\ne j}a_l G_\la (x,\xi_l)+O(\rho^2\la^2)}\,dx\\
					&=\int_{B(\xi_j,\frac{d_j}{\log\lambda})} \frac{ \e^4}
					{8\mu_j^2\rho^2}\lf(\mu_j^2\rho^2+ |x-\xi_j|^2\rg)^2 \left[1+O(\la\log\lambda|x-\xi_j|+\rho^2\la^2)\right]\,dx\\
					&=O\lf(\frac{\rho\e}{\log^4\lambda}\rg)
				\end{split}
			\end{equation*}
			In case $a_j=-1$, using previous ideas we find that
			$$\int_{B(\xi_j,\frac{d_j}{\log\lambda})} \e^2e^{U}\,dx=O\lf(\frac{\rho\e}{\log^4\lambda}\rg)$$
			and
			$$\int_{B(\xi_j,\frac{d_j}{\log\lambda})} \e^2e^{-U}\,dx=8\pi + O(\rho\la\log\lambda).$$
			Therefore, we conclude that
			$$\int_\Om \e^2(e^{U}+e^{-U})dx=8\pi m + O\lf(\rho\la\log\lambda+\frac{\rho\e}{\log^4\lambda}+\e^2\la^6\log^2\lambda\rg).$$
			
			From the choice of $\mu_j$'s in \eqref{condmu}, we obtain that
			\begin{equation}
				\begin{split}
					J_{\e,\la}(U)
					&=-16\pi m +8\pi m\log 8 -16\pi m\log(\rho\la^2)-4\pi \varphi_m(\xi)+O(\rho\la\log\lambda)\\
				\end{split}
			\end{equation}
			where $\varphi_m$ is given by \eqref{vphi}, if $\e\la^7\log\lambda|\log(\frac{\e}{\lambda})|^2$ is small enough.
		\end{proof}
		
		\begin{claim}
			The following expansion does hold
			$$\partial_{(\xi_l)_k}[J_{\e,\la}(U)]= -4\pi \partial_{(\xi_l)_k} \varphi_m(\xi) + O(\rho\la\log^2\lambda)$$
			for $l=1,\dots,m$ and $k=1,2$.
		\end{claim}
		
		\begin{proof}
			First, observe that
			$$\partial_{(\xi_l)_k}\big[J_{\e,\la}(U)\big]= D J_{\epsilon,\lambda} (U)\big[\partial_{(\xi_l)_k}U\big]=-\int_\Omega\left[\Delta U+\epsilon^2(e^{U} - e^{-U})\right]\partial_{(\xi_l)_k}U.$$
			Now, we have that
			$$\int_\Omega \partial_{(\xi_l)_k}U(-\Delta U)=\sum_{j=1}^m a_j\int_\Omega \partial_{(\xi_l)_k}U(-\Delta U_j)=\sum_{j=1}^m a_j\int_\Omega \epsilon^2e^{w_j(x)}\partial_{(\xi_l)_k}U.$$
			Using similar arguments as above and taking into account a suitable expansion for $\partial_{(\xi_l)_k}U$ (see $\partial_{(\xi'_l)_k}V$ in the proof of Proposition \ref{nlp}), we conclude that
			\begin{equation*}
				\begin{split}
					\int_\Omega \partial_{(\xi_l)_k}U(-\Delta U)=&\, -16\pi \sum_{j=1}^m \frac{\partial_{(\xi_l)_k}\mu_j}{\mu_j} + 8\pi \partial_{2k}H_\lambda(\xi_l,\xi_l) \\
					&\,+ 8\pi \sum_{\substack{j=1\\ j\ne l}}^m a_j a_l \partial_{2k}G_\lambda(\xi_j,\xi_l) + O(\rho\lambda\log^2\lambda).
				\end{split}
			\end{equation*}
			On the other hand, from similar arguments as above it follows that
			$$\epsilon^2\int_\Omega (e^{U} - e^{-U})\partial_{(\xi_l)_k}U=\epsilon^2 \sum_{j=1}^ma_j\int_\Omega (e^{U} - e^{-U})\partial_{(\xi_l)_k}U_j=O(\rho\lambda\log^2\lambda).$$
			Therefore, by using the choice of $\mu_j$ in \eqref{condmu} the claim follows . This completes the proof.
		\end{proof}
		
		\begin{claim}
			The following expansion does hold
			$$F_{\e,\la}(\xi)=J_{\e,\la}(U)+\theta_{\e,\la}^*(\xi),$$
			where
			$$
			|\theta_{\e,\la}^*(\xi)|=O(\rho^2\la^{18}\log^2\lambda|\log(\rho\la)|),
			$$
			and
			$$
			|\grad\theta_{\e,\la}^*(\xi) |=O(\rho\la^{18}\log^2\lambda|\log(\rho\la)|^2),
			$$
			as $\rho\la^{19}\to0$, uniformly on points $\xi=(\xi_1,\dots,\xi_m)\in \Om^m$ satisfying the constraints \eqref{setm}.
		\end{claim}
		
		\begin{proof}
			Since we have, $I_{\e,\la}(V)=J_{\e,\la}(U) $ and 
			$$I_{\e,\la}(V(\xi') + \phi(\xi'))
			= J_{\e,\la}(U(\xi) + \ti\phi(\xi)),$$ 
			we write
			\begin{equation*}
				\begin{split}
					J_{\e,\la}(U+\ti\phi)-J_{\e,\la}(U) & = I_{\e,\la}(V+\phi)-I_{\e,\la}(V) := A.
				\end{split}
			\end{equation*}
			Let us estimate $A$ first. Taking into account that $DI_{\e,\la}(V +
			\phi)[\phi]=0$, a Taylor expansion and
			\eqref{nline1} gives us
			\begin{equation}\label{A}
				\begin{split}
					A&= - \int_0^1 D^2I_{\e,\la}(V+t\phi)[\phi]^2\,t\,dt,\\
					&=-\int_0^1\lf( \int_{\Om_\rho} [R + N(\phi)]\, \phi\rg. \\
					&\qquad \qquad\ \ \lf. - \int_{\Om_\rho}
					\rho^2\e^2\lf[ e^{V}( e^{t \phi} -1 ) +e^{-V}(e^{-t\phi} -1)\rg] \phi^2 \rg)\,t\,dt.
				\end{split}
			\end{equation}
			Therefore, we get
			$$
			I_{\e,\la}(V+\phi)-I_{\e,\la}(V)=O(\rho^2\la^{18}\log^2\lambda|\log( \rho\la )|),
			$$
			since $\|R\|_*\le C\rho\la^9\log\lambda$, $\| N(\phi)\|_*\le
			C \|\phi\|_\infty^2$ and $\|\phi\|_\infty\le C\rho\la^9\log\lambda|\log
			(\rho\la)|$. Let us differentiate with respect to $\xi'$. We use
			representation \eqref{A} and differentiate directly under the
			integral sign, thus obtaining, for each $k=1,2$, $l=1,\dots,m$.
			Using Lemma \ref{nlp} and the computations in the proof, we conclude that for $k=1,2$, $l=1,\dots,m$
			$$
			\fr_{\xi_{kl}'}\lf[I_{\e,\la}(V+\phi)-I_{\e,\la}(V)\rg]=O(\rho^2\la^{18}\log^2\lambda|\log(\rho\la)|^2).
			$$
			Now, taking $\ti\theta_{\e,\la}(\xi')=\theta_{\e,\la}^* (\rho \xi')$ with
			$\theta_{\e,\la}^*(\xi)=F_{\e,\la} (\xi)-J_{\e,\la}(U)$, we have shown that
			$$|\ti\theta_{\e,\la} | +
			\frac{1}{|\log(\rho\la)|} |\grad_{\xi'}\ti\theta_{\e,\la} |=O(\rho^2\la^{18}\log^2\lambda|\log(\rho\la)|),
			$$
			as $\rho\lambda^9\log\lambda|\log(\rho\lambda)|^2\to0$. The continuity in $\xi$ of all these expressions is inherited from
			that of $\phi$ and its derivatives in $\xi$ in the $L^\infty$
			norm.
		\end{proof}
		
		Therefore, from \eqref{estlogrola}, previous claims and $\grad\theta_{\e,\la}^*(\xi)=\frac 1\rho \grad\ti\theta_{\e,\la}(\frac{\xi}{\rho})$ we conclude the proof of Proposition \ref{fju}.
	\end{proof}
	
	\section{Proof of main results}\label{secmain}
	
	Recall $\theta_0 \in(0,\infty)$ is the unique minima of $h$ in~\eqref{expanrobin1} in $(0,+\infty)$. We denote 
	$$
	S^* = \{ x \in \Omega : d(x) = \theta_0 \}.
	$$
	
	Concerning $h$, it is easy to see that
	\begin{equation*}
		\begin{split}
			h(\theta)&=4\log\theta +O(1) \quad\text{as }\ \theta\to +\infty, \\
			h(\theta)&=-4\log\theta +O(1) \quad\text{as }\ \theta\to 0^+.
		\end{split}
	\end{equation*}
	
	Furthermore, from \eqref{expanrobin} it follows that there is $C'=C'(K)$ such that
	$$
	\left|H_\lambda(x,x)-h(\lambda d(x)) + 4\log \la\right|\le \frac{C'}{\lambda},
	$$
	for all $x\in\Omega$ satisfying $\frac{K^{-1}}{\lambda}\le\text{dist}(x,\partial\Omega)\le \frac{K}{\lambda}$.

	\subsection{Symmetric case.} Here we follow closely the arguments in~\cite{BPW}. For this, we assume the domain $\Omega$ is such that $\Omega \cap \R \times \{ 0 \} \neq \emptyset$, and that it is symmetric with respect to the reflection at $\R \times \{ 0 \}$. In this setting, we have $G_\lambda$ in~\eqref{Green} is also symmetric with respect to this reflection in the following sense: for $x = (x_1, x_2) \in \Omega$, let $\tilde x = (x_1, -x_2)$, then
	\begin{equation}\label{sym}
		G(x,y) = G(\tilde x, \tilde y) \quad \mbox{for all} \ x \neq y.
	\end{equation}

	In fact, defining $\tilde G_\lambda(x,y) = G_\lambda(\tilde x, \tilde y)$, we have
	\begin{equation*}
		-\Delta_x \tilde G_\lambda(x,y) = -\Delta_x G_\lambda(\tilde x, \tilde y) = 8\pi\delta_{\tilde y}(\tilde x) = 8\pi \delta_{y} (x),
	\end{equation*}
	meanwhile, for each $x \in \partial \Omega$, using that $\nu(\tilde x) = (\nu(x)_1, -\nu(x)_2)$ we have
	\begin{align*}
		\RR_\lambda \tilde G_\lambda(x,y) = \nabla_x G_\lambda(\tilde x, \tilde y) \cdot (\nu(x)_1, -\nu(x)_2) + \lambda G_\lambda(\tilde x, \tilde y)
		= \RR_\lambda G(\tilde x, \tilde y) = 0,
	\end{align*}
	from which, using the uniqueness of the Green function, we conclude the symmetry property. 
	The symmetry condition is inherited by $H_\lambda$ as $H_\lambda(\tilde x) = H_\lambda(x)$. For a $m$-tupe $\xi = (\xi_1, ..., \xi_m)$, we denote $\tilde \xi = (\tilde \xi_1, ..., \tilde \xi_m)$, and 
	$\mu_{\xi, j} = \mu_{\tilde \xi, j}$ for $\mu$ as in~\eqref{condmu}, and
	$\varphi_m(\xi) = \varphi_m(\tilde \xi)$, where $\varphi$ is defined in~\eqref{vphi}.

	We stress the notation by writing $w_{\xi, j}(x), U_{\xi, j}(x)$ as the functions introduced in~\eqref{defwj},~\eqref{Uj} with the particular choice of $\xi$. 
	Then, by related invariance of the equation, it is possible to see that $w_{\tilde \xi, j}(\tilde x) = w_{\xi, j}(x)$, and that $x \mapsto U_{\tilde \xi, j}(\tilde x)$ solves the same equation than $U_{\xi,j}(x)$, from which 
	$
	U_{\xi,j}(x) = U_{\tilde \xi, j}(\tilde x) 
	$ and therefore $R_{\xi}(x) = R_{\tilde \xi}(\tilde x)$. 
	
	Now, if we denote $c(\xi) = (c_{ij}(\xi))_{ij}$ and $(c(\xi), \phi_\xi)$ as the unique solution for~\eqref{nline1} given in Proposition~\ref{nlp},
	we consider the function $\tilde \phi(x) = \phi_{\tilde \xi}(\tilde x)$, it is possible to prove that $\tilde \phi$ satisfies the same equation than $\phi_\xi$, from which they are equal by uniqueness.

	%
	%
	%
	%
	%
	%
	%
	%
	%
	%
	%
	%
	%
	%
	%
	
	Then, by definition of $F$ and using the previous symmetry properties, we have
	$
	F(\tilde \xi) = F(\xi)
	$
	from which we have $\theta(\xi) = \theta(\tilde \xi)$. Thus, we have that $F$ is a $C^1$ and symmetric with respect to $x$ perturbation of the function $\varphi$. This implies that $F$ has critical points whenever the function $\xi \mapsto \varphi((\xi_{11}, 0), (\xi_{12}, 0), \ldots, (\xi_{1m}, 0))$ has stable critical points.

	\medskip
	Now we are ready to provide the 
	\begin{proof}[Proof of Theorem~\ref{simetria}:]
		Assume $\Omega$ is simply connected, contains the origin and it is symmetric with respect to the $x$-axis.
		By symmetry and the expansion of the energy given in Proposition~\ref{fju}, it suffices to prove that there is a (nondegenerate) critical point to the function
		\begin{equation*}
			(t_1, t_2) \mapsto \varphi_2((t_1, 0), (t_2, 0)), \quad t_1, t_2 \in (a, b),  
		\end{equation*}
		where we identify the set $\{ \xi \in \Omega : \xi_2 = 0 \}$ with the interval $(a,b)$. With a slight abuse of notation, we denote this function by $\varphi_2(t_1, t_2)$. We are interested in sign-changing solutions, from which $\varphi_2$ takes the form
		\begin{equation*}
			\varphi_2(t_1, t_2) = H_\lambda((t_1,0), (t_1,0)) + H_\lambda((t_2,0), (t_2,0)) - 2G_\lambda((t_1, 0), (t_2, 0)).
		\end{equation*}
		Using the expansion~\eqref{expanrobin}, we have
		\begin{equation*}
			\varphi_2(t_1, t_2) = -8\log \lambda + h(\la d(t_1,0)) + h(\lambda d(t_2,0)) - 2G_\lambda((t_1, 0), (t_2, 0)) + O(\lambda^{-1}),
		\end{equation*}
		where $O(\lambda^{-1})$ does not depend on $(t_1, t_2)$. 
		
		Fix $\delta =(b - a)/4$ and for $K > 1$ to be determined, consider the set
		$$\Om_0=\{(t_1,t_2) \in (a,b)^2 : \lambda d((t_i, 0)) \in (K^{-1}, K) , i =1,2; \ |t_1-t_2|>\delta \}.$$
		Let us stress that our symmetry assumption implies that $(t_1,t_2)\in \partial\Omega_0$ if and only if $\lambda d((t_i,0))\in \{K^{-1},K\}$ for some $i\in\{1,2\}$. 
		%
		First, we choose $(\xi_1^*,\xi_2^*)\in\Om_0$ with $\xi_i^*\in S^*$, namely, $\la d(\xi_i^*)=\theta_0$, for $i=1,2$ so that, by the positivity of the Green's function and \eqref{expanrobin}
		\begin{equation*}
			\begin{split}
				\varphi_2(\xi_1^*,\xi_2^*)&=H_\la(\xi_1^*,\xi_1^*)+H_\la(\xi_2^*,\xi_2^*) - 2G_\la(\xi_1^*,\xi_2^*)\\
				&\le H_\la(\xi_1^*,\xi_1^*)+H_\la(\xi_2^*,\xi_2^*) \le  
				-8\log\la +2h(\theta_0) + \frac{C'}{\la}
			\end{split}
		\end{equation*}
		Let $C_\delta > 0$ such that $|G(x, y)| \leq C_\delta$ for all $x, y \in \Omega$ with $|x - y| \geq \delta$ and $d(x), d(y) \geq (\la \log \la)^{-1}$. If $\la d((t_i, 0))=K^{-1}$ for some $i$ (say $i=1$), we have that $(\xi_1,\xi_2)\in\partial\Omega_0$ and  by~\eqref{Gfxi}, we can write
		\begin{equation*}
			\begin{split}
				\varphi_2(\xi_1,\xi_2)
				&\ge -8 \log \lambda + h(K^{-1}) + h(\theta_0) - C_\delta - \frac{C'}{\la}.
			\end{split}
		\end{equation*}
		Now, we fix $K$ large just depending on $h(\theta_0)$ and $C_\delta(\Omega)$ such that
		$$
		h(K^{-1}) > C_\delta + h(\theta_0) + 2,
		$$
		which is valid for all $\la$ large enough such that $(\la \log \la)^{-1} \leq K^{-1} \la^{-1}$. Hence, choosing $\lambda$ larger, if necessary, we also get that $\frac{2C'}{\lambda}<1$.
		The same estimate can be found if $\la d(t_i, 0) = K$. 
		From here, we deduce that for any $(\xi_1, \xi_2)\in\partial\Omega_0$
		\begin{equation*}
			\begin{split}
				\varphi_2(\xi_1,\xi_2)&\ge -8\log\la +2h(\theta_0)-h(\theta_0) + h(K) -C_\delta(\Om)-\frac{C'}{\la}\\
				&\ge -8\log\la +2h(\theta_0) + \frac{C'}{\la} +1\\
				&>\varphi_2(\xi_1^*,\xi_2^*).
			\end{split}
		\end{equation*}
		This implies that $\inf_{\partial \Omega_0} \varphi_2 > \varphi_2(\xi_1^*, \xi_2^*)$, from which there is an interior minima of $\varphi_2$ in $\Omega_0$, which is stable under symmetric approximations. Therefore, by using Proposition \ref{fju}, there is an interior minima of $F_{\epsilon,\lambda}$ in $\Omega_0$ for $\epsilon>0$ small enough and $\lambda>0$ large as above. This concludes the proof.
	\end{proof}
	
	\subsection{Not simply connected case.}
	Taking advantage of previous estimates we are now ready to 
	
	\begin{proof}[Proof of Theorem~\ref{holes}:]
		Assume that $\Omega$ is not simply connected with $n$ holes $n\ge1$, so that $\partial \Omega=\cup_{i=1}^{n+1} \Gamma_i$ with $\Gamma_i$'s smooth closed curves satisfying $\Gamma_i\cap\Gamma_j=\varnothing$ for all $i\ne j$. Fix $\delta =\dfrac{a}{4}$, with 
		$$ a=\min\{\text{dist}(\Gamma_i,\Gamma_j)\ :\ i\ne j,\ i,j=1,\dots,n+1\}$$ 
		and for $K > 1$ to be determined, consider the set
		$$\Om_K=\{(\xi_1,\dots,\xi_m)\in\Omega^m\mid \lambda d(\xi_i)\in(K^{-1},K),\ i =1,2; \ |\xi_i-\xi_j|>\delta \}.$$
		Let us stress that our assumption $m\le n+1$ implies that $\xi\in \partial\Omega_K$ if and only if $\lambda d(\xi_i)\in \{K^{-1},K\}$ for some $i\in\{1,\dots,m\}$ with $\xi=(\xi_1,\dots,\xi_m)$. For simplicity we shall assume that
		$$\{1,\dots,m\}=I_1 \cup I_2,\quad \text{and}\quad i\in I_k\iff a_i=(-1)^k,\ \  k=1,2,$$
		so that $|I_k|=m_k$, $k=1,2$, $m_1+m_2=m$ and
		$$\sum_{i=1}^m \sum_{j=1\atop i\ne j} ^m a_i a_j G_\la(\xi_i,\xi_j)=\sum_{i,j\in I_1\atop i\ne j} G_\la(\xi_i,\xi_j) + \sum_{i,j\in I_2\atop i\ne j} G_\la(\xi_i,\xi_j) - 2 \sum_{i\in I_1}\sum_{j \in I_2} G_\la(\xi_i,\xi_j).$$
		In other words, we will find a sign-changing solution $u_{\varepsilon,\lambda}$ to \eqref{eq} having a positive bubble centered at $\xi_i$ with $i\in I_1$ and a negative bubble centered at $\xi_j$ with $j\in I_2$. Recall that for some $C_\delta > 0$ fixed we have that $|G(x, y)| \leq C_\delta$ for all $x, y \in \Omega$ with $|x - y| \geq \delta$ and $d(x), d(y) \geq (\la \log \la)^{-1}$. 
		Now, we choose $(\xi_1^*,\dots,\xi_m^*)\in\Om_K$ with $\xi_i\in S^*$ for all $i$, namely, $\la d(\xi_i^*)=\theta_0$, for all $i=1,\dots,m$ so that, by the positivity of the Green's function we obtain that
		\begin{equation*}
			\begin{split}
				\varphi_m(\xi_1^*,\dots,\xi_m^*)&=\sum_{i=1}^m H_\la(\xi_i^*,\xi_i^*)+ \sum_{i,j\in I_1\atop i\ne j} G_\la(\xi_i,\xi_j) + \sum_{i,j\in I_2\atop i\ne j} G_\la(\xi_i,\xi_j) \\
				&\le m\bigg[ -4\log\la + h(\theta_0) + \frac{C'}{ \la}\bigg] + \big[m_1(m_1-1) + m_2(m_2-1)\big] C_\delta\\
				&=-4m\log\la +mh(\theta_0) + \frac{m C'}{\la} \\
				&\quad + \big[m_1(m_1-1) + m_2(m_2-1)\big] C_\delta.
			\end{split}
		\end{equation*}
		On the other hand, if $\xi\in\partial\Omega$, namely, $\la d(\xi_i)=K^{-1}$ for some $i$ (say $i=1$) by~\eqref{Gfxi}, we can write
		\begin{equation*}
			\begin{split}
				\varphi_m(\xi_1,\dots,\xi_m)&\ge \sum_{i=1}^m H_\la(\xi_i,\xi_i) - 2 \sum_{i\in I_1}\sum_{j \in I_2} G_\la(\xi_i,\xi_j)\\
				&\ge (m-1)\bigg[-4\log\la + h(\theta_0) - \frac{C'}{ \la}\bigg]   -4\log\la + h(K^{-1})\\
				&\quad  - \frac{C'}{ \la} - 2m_1m_2 C_\delta\\
				&\ge -4m \log \lambda + h(K^{-1}) + (m-1) h(\theta_0) -2m_1m_2 C_\delta - \frac{m C'}{\la}.
			\end{split}
		\end{equation*}
		We fix $K$ large just depending on $h(\theta_0)$ and $C_\delta=C(\Omega)$ such that
		$$
		h(K^{-1}) > \big[m_1(m_1-1)+m_2(m_2-1)+2m_1m_2\big]C_\delta + h(\theta_0) + 2,
		$$
		which is valid for all $\la$ large enough such that $(\la \log \la)^{-1} \leq K^{-1} \la^{-1}$. Hence, choosing $\lambda$ larger, if necessary, we get that $\frac{2C'}{\lambda}<1.$
		The same estimate can be found if $\la d(\xi_i) = K$.
		From here, we deduce that for any $(\xi_1,\dots, \xi_m)\in\partial\Omega_K$
		\begin{equation*}
			\begin{split}
				\varphi_m(\xi_1,\dots,\xi_m)&\ge -4m \log \lambda + h(K^{-1}) + (m-1) h(\theta_0) -2m_1m_2 C_\delta - \frac{m C_1}{\la} \\
				&\ge-4m\log\la +mh(\theta_0) + \frac{m C_1}{\la} \\
				&\quad + \big[m_1(m_1-1) + m_2(m_2-1)\big] C_\delta +1\\
				&>\varphi_m(\xi_1^*,\dots,\xi_m^*).
			\end{split}
		\end{equation*}
		This implies that $\inf_{\partial \Omega_K} \varphi_m > \varphi_m(\xi_1^*,\dots, \xi_m^*)$, from which there is a minima of $\varphi_m$ in $\Omega_K$, which is stable under small $C^1$ perturbations. Therefore, by using Proposition \ref{fju}, there is an interior minima of $F_{\epsilon,\lambda}$ in $\Omega_K$ for $\epsilon>0$ small enough and $\lambda>0$ large as above. This concludes the proof.
	\end{proof}
	
	
	
	
	
	\bigskip
	\noindent {\bf Acknowledgements.} P. F. was partially supported by Fondecyt grant 1201884. L. I. was partially supported by Fondecyt grants 1211766 and 1221365. E. T. was partially supported by Foncecyt grant 1201897.
	

\end{document}